\documentclass[a4paper, 12pt]{article}
\usepackage[utf8]{inputenc}
\usepackage[T1]{fontenc}
\usepackage[colorlinks=true, linkcolor=blue, citecolor=blue, urlcolor=blue]{hyperref}
\usepackage[english]{babel}
\usepackage[backend=biber, bibstyle=trad-abbrv]{biblatex}
\AtEveryBibitem{
  \clearname{editor}
  \clearfield{url}
  \clearfield{doi}
  \clearfield{issn}
  \clearfield{url}
  \clearfield{urlyear}
}
\usepackage{geometry}

\usepackage{amsmath, amssymb, amsthm}
\usepackage{mathtools}
\usepackage{titling}
\usepackage{float}
\usepackage{mathrsfs}
\usepackage{dsfont}
\usepackage{upgreek}
\usepackage{tcolorbox}
\usepackage{xcolor}
\usepackage{extarrows}
\usepackage{caption}
\usepackage{array}
\usepackage{subcaption}
\usepackage{stmaryrd}
\usepackage{tikz}
\usepackage{tikz-3dplot}
\usetikzlibrary{calc, cd, patterns}
\usepackage[ruled, linesnumbered]{algorithm2e}
\usepackage{multicol}
\usepackage{csquotes}
\usepackage{enumitem}
\usepackage{titlesec}
\usepackage{titletoc}
\usepackage{aliascnt}
\usepackage{cleveref}
\makeatletter
\DeclareFontFamily{OMX}{MnSymbolE}{}
\DeclareSymbolFont{MnLargeSymbols}{OMX}{MnSymbolE}{m}{n}
\SetSymbolFont{MnLargeSymbols}{bold}{OMX}{MnSymbolE}{b}{n}
\DeclareFontShape{OMX}{MnSymbolE}{m}{n}{
    <-6>  MnSymbolE5
  <6-7>  MnSymbolE6
  <7-8>  MnSymbolE7
  <8-9>  MnSymbolE8
  <9-10> MnSymbolE9
  <10-12> MnSymbolE10
  <12->   MnSymbolE12
}{}
\DeclareFontShape{OMX}{MnSymbolE}{b}{n}{
    <-6>  MnSymbolE-Bold5
  <6-7>  MnSymbolE-Bold6
  <7-8>  MnSymbolE-Bold7
  <8-9>  MnSymbolE-Bold8
  <9-10> MnSymbolE-Bold9
  <10-12> MnSymbolE-Bold10
  <12->   MnSymbolE-Bold12
}{}

\let\llangle\@undefined
\let\rrangle\@undefined
\DeclareMathDelimiter{\llangle}{\mathopen}%
                    {MnLargeSymbols}{'164}{MnLargeSymbols}{'164}
\DeclareMathDelimiter{\rrangle}{\mathclose}%
                    {MnLargeSymbols}{'171}{MnLargeSymbols}{'171}
\makeatother

\DeclareMathAlphabet{\mathcal}{OMS}{zplm}{m}{n}
\DeclareFontFamily{U}{stix2bb}{\skewchar\font127 }
\DeclareFontShape{U}{stix2bb}{m}{n} {<-> stix2-mathbb}{}
\DeclareMathAlphabet{\mathbbl}{U}{stix2bb}{m}{n}

\newcommand{\KL}{\mathrm{KL}}
\newcommand{\OT}{\mathrm{OT}}
\newcommand{\supp}{\mathrm{supp}}
\newcommand{\id}{\mathrm{Id}}
\newcommand{\diag}{\mathrm{diag}}
\renewcommand{\div}{\mathrm{div}}
\newcommand{\inv}{^{-1}}

\DeclareMathOperator*{\argmin}{arg\,min}
\makeatletter
\let\olddot\Dot
\makeatother
\renewcommand{\Dot}[1]{\smash{\olddot{#1}}}

\renewcommand{\tau}{\uptau}
\newcommand{\dd}{\mathrm d}
\DeclareMathOperator{\Tr}{Tr}

\newcommand{\bR}{\mathbb R}
\newcommand{\mP}{\mathcal P}

\newtheoremstyle{sl}
{3pt}
{3pt}
{\slshape}
{}
{\bfseries}
{.}
{.5em}
{}
\theoremstyle{sl}
\newtheorem{theorem}{Theorem}[section]

\newaliascnt{lemma}{theorem}
\newaliascnt{proposition}{theorem}
\newaliascnt{corollary}{theorem}
\newaliascnt{definition}{theorem}
\newaliascnt{remark}{theorem}
\newaliascnt{example}{theorem}

\newtheorem{lemma}[lemma]{Lemma}
\aliascntresetthe{lemma}

\newtheorem{proposition}[proposition]{Proposition}
\aliascntresetthe{proposition}

\aliascntresetthe{corollary}

\newtheorem{definition}[definition]{Definition}
\aliascntresetthe{definition}

\theoremstyle{remark}
\newtheorem{remark}[remark]{Remark}
\aliascntresetthe{remark}

\aliascntresetthe{example}

\floatstyle{boxed}
\newfloat{listing}{H}{loi}
\floatname{listing}{Pseudocode}
\newlist{thmenum}{enumerate}{1}
\setlist[thmenum]{label=(\alph*), ref=(\alph*)}

\numberwithin{equation}{section}
\crefrangeformat{equation}{(#3#1#4--#5#2#6)}
\crefmultiformat{equation}{(#2#1#3)}%
{ and~(#2#1#3)}{, (#2#1#3)}{ and~(#2#1#3)}
\crefformat{equation}{(#2#1#3)}
\crefname{theorem}{Theorem}{Theorems}
\crefname{lemma}{Lemma}{Lemmas}
\crefname{proposition}{Proposition}{Propositions}
\crefname{corollary}{Corollary}{Corollaries}
\crefname{algorithm}{Algorithm}{Algorithms}
\crefname{definition}{Definition}{Definitions}
\crefname{assumption}{Assumption}{Assumptions}
\crefname{remark}{Remark}{Remarks}
\crefname{example}{Example}{Examples}
\crefname{pcode}{Pseudocode}{Pseudocodes}
\newcommand{\refthmitem}[2]{\cref{#1} \ref{#1:#2}}
\crefformat{figure}{Figure #2#1#3}
\crefformat{section}{Section #2#1#3}
\crefformat{subsection}{Section #2#1#3}
\crefformat{appendix}{Appendix #2#1#3}
\newcommand{\theaffiliations}{}
\newcommand{\affiliations}[1]{\renewcommand{\theaffiliations}{#1}}
\let\oldmaketitle\maketitle
\renewcommand{\maketitle}{
\renewcommand{\thefootnote}{\fnsymbol{footnote}}
\oldmaketitle \theaffiliations
\renewcommand{\thefootnote}{\arabic{footnote}}}
\author{Mathis Hardion\protect\footnotemark[1], Théo Lacombe\protect\footnotemark[1]}
\affiliations{\footnotetext[1]{Laboratoire d’Informatique Gaspard Monge,
Univ. Gustave Eiffel, CNRS, LIGM, F-77454
Marne-la-Vallée, France.\\Email addresses: \texttt{mathis.hardion2@univ-eiffel.fr, theo.lacombe@univ-eiffel.fr}}}
\title{The Wasserstein gradient flow of the Sinkhorn divergence between Gaussian distributions}
\date{\today}

\begin{document}
\maketitle

\begin{abstract}
    We study the Wasserstein gradient flow of the Sinkhorn divergence when both the source and the target are Gaussian distributions. 
    We prove the existence of a flow that stays in the class of Gaussian distributions, and is unique in the larger class of measures with strongly-concave and smooth log-densities. We prove that the flow globally converges toward the target measure when the source's covariance matrix is not singular, and provide counter-examples to global convergence when it is, giving a first answer to an open question raised in \cite[\S 4.2]{carlier2024displacement}. 
    When the covariance matrix of the source distribution commutes with that of the target, we derive more quantitative results that showcase exponential convergence toward the target when the source and the target share their support, but dropping to sublinear rates (\(O(t^{-1})\)) if the target is concentrated on a strict subspace of the source's support.

\medskip


\medskip

\end{abstract}

\section{Introduction}

In this work, we study the Wasserstein gradient flow of the Sinkhorn divergence (a debiased version of entropic optimal transport) to a target measure, in the context of Gaussian distributions. This restriction allows to obtain closed form expressions, with the goal of studying the convergence of the flow to the target measure, giving a partial answer to the question raised initially in \cite[\S 4.2]{carlier2024displacement}. 

On \(\bR^d\) with its Euclidean norm \(\|\cdot\|\), denoting \(\mP_2(\bR^d)\) the set of Radon probability measures on \(\bR^d\) with finite second moment, the Sinkhorn divergence with regularization parameter \(\varepsilon > 0\) is defined for \(\mu, \nu \in \mP_2(\bR^d)\) as
\begin{equation*}
S_\varepsilon(\mu,\nu) \coloneqq \OT_\varepsilon(\mu,\nu) - \frac12\OT_\varepsilon(\mu,\mu) - \frac12\OT_\varepsilon(\nu,\nu).
\end{equation*}
Here the entropic optimal transport cost is defined by the Schr\"odinger problem:
\begin{align*}
\OT_\varepsilon(\mu, \nu) &\coloneqq \min_{\pi\in\Pi(\mu,\nu)}\int\|x-y\|^2\dd \pi(x,y) + \varepsilon\KL(\pi|\mu\otimes\nu)
\end{align*}
where the KL divergence is defined through
\(\KL(\pi|\gamma) \coloneqq \int\log\left(\frac{\dd \pi}{\dd \gamma}\right)\dd \pi\) if \(\pi \ll \gamma\) and \(+\infty\) otherwise, and \(\Pi(\mu,\nu)\) is the set of transport plans between \(\mu\) and \(\nu\), i.e.~probability measures on \(\bR^d\times\bR^d\) with marginals \((\mu, \nu)\). This problem has an equivalent dual formulation:
\begin{equation*}
\OT_\varepsilon(\mu, \nu)=\max_{f,g}\int f\dd \mu + \int g\dd \nu - \varepsilon\int(\exp(\tfrac1\varepsilon(f\oplus g - c))-1)\dd (\mu\otimes\nu).
\end{equation*}
where the maximum is taken over continuous bounded functions, \(f\oplus g \colon (x,y) \mapsto f(x) + g(y)\) and we denote for conciseness \(c:(x,y)\mapsto \|x-y\|^2\). 
This dual problem admits maximizers denoted \((f_{\mu,\nu}, g_{\mu,\nu})\) and called Schr\"odinger potentials. 
They are unique up to constant shifts \((f_{\mu,\nu} + \lambda, g_{\mu,\nu} - \lambda)\) for \(\lambda\in\bR\). 
When \(\mu=\nu\), one can choose \(\lambda\) such that \(f_{\mu,\mu} = g_{\mu,\mu}\), a choice that we will denote by \( f_\mu \) to alleviate notation. At optimality, the cost is given by \(\OT_\varepsilon(\mu,\nu) = \int f_{\mu,\nu}\dd\mu + \int g_{\mu,\nu}\dd\nu\).

The Wasserstein distance corresponds to the unregularized optimal transport problem, i.e.~it is defined as \(W_2 = \sqrt{\OT_0}\). It induces a differential structure on \(\mP_2(\bR^d)\), from which a notion of gradient flow may be derived \cite{ambrosio2008gradient}. Generally speaking, one can expect the Wasserstein gradient flow of a sufficiently regular functional \(F:\mP_2(\bR^d) \to \bR\) to be written as a solution \((\mu_t)_t\) of the continuity equation
\begin{equation} \label{eq:CE-WGF-nabladeltaF}
\dot \mu_t = \div\left(\mu_t\nabla\frac{\delta F}{\delta\mu_t}\right)
\end{equation}
where \(\frac{\delta F}{\delta \mu}\) denotes the first variation of \(F\) at \(\mu\in \mP_2(\bR^d)\), i.e.~the function from \(\bR^d\) to \(\bR\) such that for any signed measure \(\sigma = \nu - \mu\) with \(\nu \in \mP_2(\bR^d)\), \(\frac{\dd}{\dd t}\big|_{t=0} F(\mu + t\sigma) = \int \frac{\delta F}{\delta \mu}\dd \sigma\).

In this work, we thus consider the functional \(S_\varepsilon(\cdot, \mu_\star)\) for a prescribed target Gaussian measure \(\mu_\star\), and study its gradient flow in the Wasserstein geometry for a given initial Gaussian measure \(\mu_0\). Under assumptions of subgaussianity (and in particular for compactly supported measures), the first variation of \(S_\varepsilon(\cdot, \mu_\star)\) at \(\mu\) is given by \(f_{\mu,\mu_\star} - f_{\mu}\) \cite{janati2020debiased, feydy2019interpolating}. Thus, we can expect the gradient flow we are interested in to be a solution of the continuity equation (which we dub SWGF for Sinkhorn Wasserstein Gradient Flow)
\begin{equation} \label{eq:SWGF}
\dot \mu_t = \div(\mu_t \nabla(f_{\mu_t,\mu_\star} - f_{\mu_t})). \tag{SWGF}
\end{equation}
In the setting of compactly-supported measures, it is proven in \cite[Section 4.2]{carlier2024displacement} that this PDE indeed corresponds to the gradient flow of \(S_\varepsilon(\cdot, \mu_\star)\) in the sense of \cite{ambrosio2008gradient}, however, to the extent of our knowledge this result has not been rigorously proved for measures with unbounded support.

Since \(S_\varepsilon(\mu, \mu_\star) \geq 0\) with equality if and only if \(\mu = \mu_\star\) \cite{feydy2019interpolating}, and that the PDE \eqref{eq:SWGF} should be a gradient flow i.e.~a curve of steepest descent of the functional \(S_\varepsilon(\cdot, \mu_\star)\) in some sense, it is natural to ask whether \(\mu_t\) converges toward the target \(\mu_\star\) as \(t\to+\infty\), a question initially raised by \cite{carlier2024displacement}.
Our first partial answer to this question is given by the following summary of \cref{thm:µflow}, \cref{prop:subdiff-Seps-G}, \cref{thm:covflow} and \cref{prop:eigflow} below.
\begin{theorem}\label{thm:cv-criterion}
Let \(\mu_0, \mu_\star\) be Gaussian measures, and \(\supp(\mu_0)\), \(\supp(\mu_\star)\) their respective supports. There exists a unique solution \((\mu_t)_t\) of \eqref{eq:SWGF} which stays Gaussian, it is a Wasserstein gradient flow of $S_\varepsilon(\cdot, \mu_\star)$ in the sense of \cite{ambrosio2008gradient}, and
\begin{thmenum}
\item If \(\mu_0\) is non-singular, then \(\mu_t\) converges in Wasserstein distance to \(\mu_\star\) as \(t\to+\infty\). \label{thm:cv-criterion:ns}
\item If \(\mu_0\) is singular and \(\mu_\star\) is non-singular, then \(\mu_t\) converges in Wasserstein distance to a measure \emph{different} from \(\mu_\star\). \label{thm:cv-criterion:s2ns}
\item If the covariances of \(\mu_0\) and \(\mu_\star\) commute, then \(\mu_t\) converges to \(\mu_\star\) if and only if \(\supp(\mu_\star) \subseteq \supp(\mu_0)\). In that case, the functional converges to 0 with an exponential rate if \(\supp(\mu_0) = \supp(\mu_\star)\), and only a sublinear rate when \(\supp(\mu_\star)\) is of strictly lower dimension than \(\supp(\mu_0)\). \label{thm:cv-criterion:commute}
\end{thmenum}
\end{theorem}

\paragraph{Outline and Contributions.} 
After reviewing the basic necessary tools in \cref{sec:background}, we show in \cref{sec:wellposed:existence} that the restriction of \(S_\varepsilon(\cdot, \mu_\star)\) to Gaussian distributions is convex along generalized geodesics (with a negative convexity constant), yielding in particular the existence of a curve satisfying \eqref{eq:SWGF} and its uniqueness among Gaussian curves. We prove in \cref{sec:wellposed:WGFcharacterization} that this curve is indeed a Wasserstein gradient flow of the original functional in the whole space $\mP_2(\bR^d)$, without the need to restrict to Gaussian distributions.
Using an Evolution Variational Inequality (EVI) argument, we obtain in \cref{sec:wellposed:EVI} the uniqueness of this curve in the larger set of measures with strongly concave and smooth log-densities. 
In \cref{sec:equation-cv}, leveraging the explicit formulation of the Sinkhorn divergence (and its gradient) for Gaussian distributions, we prove global convergence toward \(\mu_\star\) of the flow when the source measure has a non-singular covariance matrix, and otherwise give a precise characterization of the limit points. Assuming furthermore that the covariance matrices of the source and target distributions commute, we derive in \cref{sec:equation-cv:commute} an evolution equation of the eigenvalues of the covariances throughout the flow, yielding exponential convergence rates when both the source and target have the same support, and only sublinear ones when the target is concentrated on a strict subspace of the source's support.
Eventually, in \cref{sec:numerics}, we show that an explicit Euler scheme can be used to faithfully approximate the flow, yielding a straightforward way to run simulations. 
We therefore numerically confirm the tightness of the theoretical rates derived in the previous section. We also explore a non-Gaussian example, which suggests that the deterioration of the convergence rate when the target's support is of lower dimensionality than the source's support may hold in more general settings.

\subsection{Related work} 
\paragraph{Optimal transport and its entropic regularization.} The Wasserstein distance is a tool to compare probability distributions in a way that is aware of the geometry of the underlying space, with many applications in data science, probability and statistics, modeling, economics and more. 
We refer the reader to the textbooks and surveys \cite{villani2009optimal, santambrogio2015optimal, peyre2019computational, chewi2025statistical, ambrosio2013users, rachev1998mass, rachev1998massa} and references therein for an overview of the breadth of this theory and its applications. 
One of the main shortcomings of the Wasserstein distance is its high computational cost, and a popular remedy is entropic regularization.
The Entropic Optimal Transport (EOT) problem can be traced back to Schr\"odinger \cite{schrodinger1932theorie}, and has been popularized by \cite{cuturi2013sinkhorn} which showed it could be solved efficiently on GPU thanks to parallelization, making it a great approximation of the Wasserstein distance for large-scale applications in machine learning. The entropic optimal transport cost was found to have extra desirable properties, such as better statistical properties \cite{genevay2019sample, mena2019statistical} and smoothness with computationally accessible gradients \cite{genevay2018learning}. We refer to \cite{leonard2014survey, nutz2022introduction} for more on the theory of EOT and to \cite{peyre2019computational} for computational aspects. The main disadvantage of this regularized cost is that in contrast to unregularized OT, one has \(\OT_\varepsilon(\mu,\mu) > 0\) and---arguably worse---the minimizer of the functional \(\mu \mapsto \OT_\varepsilon(\mu,\mu_\star)\) is not \(\mu_\star\), but instead a ``shrunk'' version of it. The Sinkhorn divergence was thus introduced in \cite{ramdas2017wasserstein,genevay2018learning} (see also \cite{salimans2018improving}) to overcome this peculiar behavior, by adding the negative self-entropic term \(-\frac12\OT_\varepsilon(\mu, \mu)\) which behaves like a repulsive interaction term when minimizing the functional, with the goal of compensating the shrinking effect. And indeed, the work \cite{feydy2019interpolating} has proven that \(S_\varepsilon\) is a positive definite loss, which is also smooth and metrizes convergence in law/weak-\textasteriskcentered{} convergence, giving it a firm theoretical ground to motivate its use in applications.

\paragraph{Wasserstein gradient flows.} Gradient flows in Euclidean space are ODEs written as \(\dot x_t = -\nabla F(x_t)\) for some regular enough \(F\), and are a natural way to minimize such functionals (the discretization of a gradient flow is gradient descent). It is possible to extend this notion to metric spaces, and in particular to the Wasserstein space \( (\mP_2(\bR^d), W_2 ) \). The pioneering work \cite{jordan1998variational} shows with a variational scheme (now called JKO in reference to the authors) that the Fokker-Planck equation can be seen as a gradient flow of a free energy (that is, the sum of a potential energy \(\mu\mapsto \int V\dd \mu\) and the negative Boltzmann entropy) with respect to the Wasserstein distance. This has sparked an interest in the Wasserstein geometry to describe PDEs, including the porous medium equation studied in \cite{otto2001geometry} which developed the ``Otto calculus'', a formal Riemannian structure on the Wasserstein space, later developed into a more rigorous theory in \cite{ambrosio2008gradient}. Since then, Wasserstein gradient flows have gained traction in applied fields as well, as a modeling tool \cite{maury2010macroscopic, bunne2022proximal}, a sampling method \cite{sun2023discrete}, or to study learning algorithms \cite{chizat2018global} and many more. We refer to the textbooks \cite{ambrosio2008gradient, santambrogio2015optimal, peyre2019computational, chewi2025statistical} and surveys \cite{santambrogio2017euclidean,ambrosio2013users} and references therein for an overview of this theory and its applications.

\paragraph{The gradient flow of \texorpdfstring{\(S_\varepsilon\)}{Sε}.} The idea of considering the Wasserstein gradient flow of the Sinkhorn divergence is not new. The seminal paper \cite{feydy2019interpolating} includes numerical experiments on this flow (on particles), which is mentioned as a non-parametric data fitting model. Multiple works have also applied this flow to generative modeling: \cite{zhu2024neural} implements neural networks to approximate the velocity field of the flow and pushes samples along it at inference, while more recent works \cite{he2026sinkhorndrifting, han2026onestep} have utilized it within the framework of drifting models \cite{deng2026generative} which is one-step at inference (see also \cref{sec:numerics:swiss-roll} for a more detailed description).
From a more theoretical standpoint, the well-posedness of the flow on compactly-supported measures was studied in \cite{carlier2024displacement}, where the question of the convergence to the target measure was raised. The authors explicitly state that their results rely heavily on this compactness assumption, and to the best of our knowledge our study of this problem for Gaussian distributions is the first to obtain results without compactness. 

\paragraph{Distinction from Sinkhorn geodesics, barycenters, and Schr\"odinger bridges.}
The flow that we study, under the conditions explained in \cref{thm:cv-criterion}, can interpolate between source and target measures. There are other ways to interpolate between measures using entropic optimal transport, such as Sinkhorn barycenters \cite{janati2020debiased}, Sinkhorn geodesics \cite{lavenant2025riemannian}, and Schr\"odinger bridges \cite{leonard2014survey}. When \(\varepsilon=0\), all those notions coincide, including the gradient flow up to a time reparametrization \(t\mapsto 1-e^{-2t}\). However for \(\varepsilon >0\), there is a priori no reason they do. Firstly, the Schr\"odinger bridge with positive temperature is not constant when its extremities are the same, in opposition to the other curves. Secondly, the Sinkhorn barycenters and Sinkhorn geodesics are built solely on \(S_\varepsilon\) and its induced geometry, whereas our gradient flow incorporates the Wasserstein distance as underlying geometry and \(S_\varepsilon\) as functional. Finally, there are cases where the gradient flow does not converge to the target, whereas all the other curves can interpolate between any pair of measures. On the other hand, the PDE \eqref{eq:SWGF} has the benefit of being easily tractable (e.g.~for Gaussian distributions, or for discrete measures) from a computational standpoint, while the aforementioned alternatives yield way more challenging optimization problems.

\paragraph{Optimal transport and optimization on the Bures--Wasserstein space.}
Optimal transport between Gaussian measures is simpler thanks to closed form expressions. In the unregularized case \(\varepsilon = 0\) we recover the Bures--Wasserstein distance, which induces a Riemannian structure on the space of covariance matrices \cite{bhatia2019bures, takatsu2011wasserstein, malago2018wasserstein}. The Sinkhorn divergence too admits a closed form expression in this case \cite{janati2020entropic, mallasto2022entropyregularized}, see also \cref{sec:background:preliminary-gaussian}. While working in the Bures--Wasserstein space (i.e.~the space of Gaussian measures endowed with \(W_2\)) may seem restrictive, there is a growing body of work studying optimization and gradient flows on that space, notably for variational inference \cite{lambert2022variational, diao2023forwardbackward, diao2023proximal, yi2023bridging}, positive semidefinite optimization \cite{han2021riemannian,maunu2023bureswasserstein}, and robotics \cite{ziesche2023wasserstein}.

\subsection{Notations and setting}
We work in the Wasserstein space \((\mP_2(\bR^d), W_2)\) of Radon probability measures with finite second moment endowed with the Wasserstein distance. The inner product of \(\bR^d\) is denoted \(\langle\cdot,\cdot\rangle\). The set of optimal transport plans between \(\mu, \nu\in\mP_2(\bR^d)\) is denoted \(\Pi_{\mathrm o}(\mu, \nu) \coloneqq \argmin_{\pi\in\Pi(\mu,\nu)}\int\|x-y\|^2\dd \pi(x,y)\). We denote \(\mathbf p_i:(\bR^d)^2 \to \bR^d\) the projection onto the \(i\)th component of the product space for \(i\in\{1, 2\}\), i.e.~\(\mathbf p_i(x_1, x_2) = x_i\) for \((x_1, x_2) \in (\bR^d)^2\). Similarly, on \((\bR^d)^3\) we define \(\mathbf p_{i,j}(x_1, x_2, x_3) = (x_i, x_j)\). We denote \(L_\mu^2(\bR^d)\) the set of vector fields \(\bR^d \to \bR^d\) with integrable square norm with respect to \(\mu\in\mP_2(\bR^d)\), quotiented by equality \(\mu\)-almost everywhere. It is endowed with its usual inner product \(\langle\cdot, \cdot\rangle_{L^2_\mu(\bR^d)}\) and corresponding norm \(\|\cdot\|_{L^2_\mu(\bR^d)}\). The set of symmetric, positive semi-definite (resp. positive definite) \(d\times d\) matrices is written \(\mathbb S^d_{+}\) (resp. \(\mathbb S^d_{++}\)), and the operator norm (equal to the spectral radius for symmetric matrices) is denoted by \(\|\cdot\|_{\mathrm{op}}\) . For \(m \in \bR^d\), \(\Sigma\in\mathbb S_+^d\) we denote \(\mathcal N(m, \Sigma)\) the (possibly singular) Gaussian of mean \(m\) and covariance \(\Sigma\), i.e.~the measure of characteristic function given by \(s\mapsto \exp(i\langle s, m\rangle - \frac12\langle s, \Sigma s\rangle)\). The largest eigenvalue of a covariance matrix \( \Sigma \) is denoted by \( \lambda_{\max}(\Sigma)  \in \bR_+ \), and the smallest is written \(\lambda_{\min}(\Sigma)\).

\section{Background and preliminary results}\label{sec:background}

\subsection{Wasserstein Gradient Flows}\label{sec:background:WGFs}
We now give a brief overview of the tools developed in \cite{ambrosio2008gradient}, see also the surveys \cite{santambrogio2017euclidean, ambrosio2013users}. This technical formalism yields powerful existence results for gradient flows that we will make use of. Note also that while \eqref{eq:CE-WGF-nabladeltaF} is often used directly as a definition of gradient flow, if one were to use it as such, one would first need to prove that $\frac{\delta S_\varepsilon(\cdot,\mu_\star)}{\delta\mu} = f_{\mu,\mu_\star} - f_\mu$ for \emph{all} \( \mu \in \mP_2(\bR^d)\),

which would prove challenging from the lack of regularity of this space (the most general version of this result to the best of our knowledge holds when $\mu,\mu_\star$ are subgaussian measures \cite{janati2020debiased}).

The Wasserstein subdifferential of a functional \(F:\mP_2(\bR^d) \to \bR\cup\{+\infty\}\) is defined as follows: we write \(\xi \in \partial_W F(\mu)\) if for any \(\nu\) there exists \(\pi\in\Pi_{\mathrm o}(\mu,\nu)\) such that
\begin{equation}\label{eq:defsubdiff}
F(\nu) - F(\mu) \geq \int\langle\xi(x), y-x\rangle \dd \pi(x, y) + o(W_2(\mu,\nu)).
\end{equation}
This definition may be seen as an analogue of the Fr\'echet subdifferential adapted to the Wasserstein space.
A Wasserstein gradient flow in the sense of \cite{ambrosio2008gradient} is then defined as a curve \((\mu_t)_t\) satisfying a continuity equation \(\dot \mu_t + \div(\mu_tv_t)=0\), where \(v_t\in-\partial_WF(\mu_t)\) for a.e.~\(t\). More details on the pseudo-Riemannian structure of the Wasserstein space motivating this definition are provided in \cref{sec:pseudoriemW2}.
Within the subdifferential, the element of minimal \(L^2\) norm, denoted \(\partial_W^\circ F(\mu) \coloneqq \argmin\{\|\xi\|_{L^2_\mu(\bR^d)} \: : \: \xi\in \partial_WF(\mu)\}\) plays a distinguished role. Indeed, gradient flows satisfy a \emph{minimal selection principle} \cite[Theorem 11.1.3]{ambrosio2008gradient}, meaning that if a gradient flow exists then it is a solution of \(\dot\mu_t = \div(\mu_t \partial_W^\circ F(\mu_t))\). In most known cases, we have \(\partial_WF(\mu) = \{\nabla \frac{\delta F}{\delta\mu}\}\), which explains why \eqref{eq:CE-WGF-nabladeltaF} is a popular alternative definition of Wasserstein gradient flow.

A key assumption of \cite{ambrosio2008gradient} guaranteeing existence and uniqueness of a flow is the \(\lambda\)-convexity along geodesics of the functional \(F\), namely: for all \(\mu, \nu \in \mP_2(\bR^d)\), there exists a geodesic \((\mu_t)_{t\in[0, 1]}\) between \(\mu\) and \(\nu\) such that for all \(t\),
\begin{equation*}
F(\mu_t) \leq (1-t)F(\mu)  +tF(\nu) - \frac\lambda2 t(1-t)W_2^2(\mu, \nu).
\end{equation*}
We recall that Wasserstein geodesics are written \(\mu_t = ((1-t)\mathbf p_1 + t\mathbf p_2)_\#\pi\) for \(\pi\in\Pi_{\mathrm o}(\mu, \nu)\). There is also the stronger notion of \(\lambda\)-convexity along \emph{generalized} geodesics: for \(\omega\in\mP_2(\bR^d)\), a generalized geodesic between \(\mu\) and \(\nu\) of base \(\omega\) is given by \(\mu_t = ((1-t)\mathbf p_{1, 2} + t \mathbf p_{1, 3})_\#\bar \pi\) for \(\bar \pi \in \mP((\bR^d)^3)\) with marginals \((\omega, \mu, \nu)\) and \((\mathbf p_{1, 2})_\#\bar\pi \in \Pi_{\mathrm o}(\omega, \mu)\), \((\mathbf p_{1, 3})_\#\bar\pi \in \Pi_{\mathrm o}(\omega, \nu)\). A functional \(F\) is said to be \(\lambda\)-convex along generalized geodesics if for any \((\omega, \mu, \nu)\), there exists a generalized geodesic \((\mu_t)_{t\in[0, 1]}\) of base \(\omega\) between \(\mu\) and \(\nu\) such that
\begin{equation*}
F(\mu_t) \leq (1-t)F(\mu)  +tF(\nu) - \frac\lambda2 t(1-t)W_{\bar\pi}^2(\mu, \nu)
\end{equation*}
where \(W_{\bar\pi}^2(\mu, \nu) \coloneqq \int\|y-z\|^2\dd\bar\pi(x, y, z)\).
Under either of these assumptions, we also obtain an Evolution Variational Inequality (EVI) for gradient flows, that is if \((\mu_t)_t\) is a Wasserstein gradient flow of \(F\) then
\begin{equation}\label{eq:EVIdef}
\forall \nu \in \mP_2(\bR^d), \frac{\dd}{\dd t}\frac12W_2^2(\mu_t, \nu) \leq F(\nu) - F(\mu_t) - \frac\lambda2 W_2^2(\mu_t, \nu)
\end{equation}
which yields uniqueness of solutions for a fixed initial measure. Indeed, for two curves \((\mu_t^1)_t\) and \((\mu_t^2)_t\) satisfying the EVI, writing \eqref{eq:EVIdef} with \((\mu_t, \nu) = (\mu_t^1, \mu_t^2)\) then \((\mu_t, \nu) = (\mu_t^2, \mu_t^1)\) and summing the two inequalities gives
\begin{equation*}
\frac{\dd}{\dd t}\frac12 W_2^2(\mu_t^1, \mu_t^2) \leq -\lambda W_2^2(\mu_t^1, \mu_t^2).
\end{equation*}
By Gronwall's lemma, we deduce \(W_2^2(\mu_t^1, \mu_t^2)\leq e^{-2\lambda t}W_2^2(\mu_0^1, \mu_0^2)\). In particular, when \(\mu_0^1=\mu_0^2\) then \(\mu_t^1 = \mu_t^2\) for a.e.~\(t\).

We obtain such an inequality in \cref{sec:wellposed:EVI} to obtain uniqueness on the class of measures with strongly concave and smooth log-densities.

\subsection{Preliminary results on EOT between Gaussian measures}\label{sec:background:preliminary-gaussian}
We now introduce the results we need in the specific case of Gaussian measures, where the Sinkhorn divergence and the Schr\"odinger potentials have an explicit formula as a function of the means and covariances \cite{janati2020entropic}. We slightly extend these results to the case of singular covariance matrices, and also give the expression of the Schr\"odinger potentials for uncentered measures.
\begin{proposition}\label{thm:Seps-gauss}
Let \(\Sigma, \Gamma \in\mathbb S^d_+\), \(m, n\in\bR^d\), and \(\mu = \mathcal N(m, \Sigma)\), \(\nu = \mathcal N(n, \Gamma)\). We have
\begin{equation*}
S_\varepsilon(\mu, \nu) = \|m-n\|^2 + \mathcal B_\varepsilon(\Sigma, \Gamma)
\end{equation*}
where
\begin{align*}
\mathcal B_\varepsilon(\Sigma, \Gamma) &\coloneqq B_\varepsilon(\Sigma, \Gamma) - \frac12 B_\varepsilon(\Sigma, \Sigma) - \frac12 B_\varepsilon(\Gamma, \Gamma),\\
B_\varepsilon(\Sigma, \Gamma) &\coloneqq \Tr(\Sigma) + \Tr(\Gamma) +\tfrac\varepsilon2\log\det(D_\varepsilon +\tfrac\varepsilon2\id)- \Tr(D_\varepsilon),\\
D_\varepsilon &\coloneqq (4\Sigma^{\frac12}\Gamma\Sigma^{\frac12}+\tfrac{\varepsilon^2}4\id)^{\frac12}.
\end{align*}
The Schr\"odinger potentials are given up to constant shifts by \(f_{\mu,\nu}(x) = \langle x-m, F_\Sigma^\Gamma(x-m)\rangle + 2\langle m-n, x\rangle\) and \(g_{\mu,\nu}(x) = \langle x-n, F_\Gamma^\Sigma(x-n)\rangle + 2\langle n-m, x\rangle - \|m-n\|^2\) where, denoting \(\tilde \varepsilon \coloneqq \frac\varepsilon4\),
\begin{equation*}
F_\Sigma^\Gamma \coloneqq \id - \Gamma ^{\frac12}((\Gamma^{\frac12} \Sigma\Gamma^{\frac12} + \tilde\varepsilon^2\id)^{\frac12}+\tilde\varepsilon\id)\inv\Gamma ^{\frac12}.
\end{equation*}
Thus \(\nabla(f_{\mu, \nu} - f_{\mu})(x) = G_{\Sigma}^{\Gamma}(x-m) + 2(m-n)\) with
\begin{equation*}
G_{\Sigma}^{\Gamma} \coloneqq 2(F_\Sigma^\Gamma-F_\Sigma^\Sigma) =2\Sigma((\Sigma^2 + \tilde\varepsilon^2\id)^{\frac12}+\tilde{\varepsilon}\id)\inv - 2\Gamma ^{\frac12}((\Gamma^{\frac12} \Sigma\Gamma^{\frac12} + \tilde\varepsilon^2\id)^{\frac12}+\tilde\varepsilon\id)\inv\Gamma ^{\frac12},
\end{equation*}
and finally \(\nabla_1\mathcal B_\varepsilon(\Sigma, \Gamma) = \frac12G_{\Sigma}^{\Gamma}\).
\end{proposition}
In the case of non-singular covariances, this result is a direct consequence of  \cite[Theorem 1 and Corollary 1]{janati2020entropic}, which deal with centered measures, and of the following lemma giving the Schr\"odinger potentials in the uncentered case. The extension to the singular case is proved afterwards.
\begin{lemma}\label{lem:uncentered-potentials}
Let \(\mu, \nu\in \mP_2(\bR^d)\), \(m, n\in \bR^d\) their respective means, and \(\Bar \mu, \Bar \nu\) their respective centered versions, then we have (up to constant shifts)
\begin{align*}
f_{\mu, \nu} &= f_{\Bar \mu, \Bar \nu}(\cdot - m) + 2 \langle m - n, \cdot \rangle\\
g_{\mu, \nu} &= g_{\Bar \mu, \Bar \nu}(\cdot - n) + 2 \langle n - m, \cdot \rangle - \| m-n\|^2.
\end{align*}
\end{lemma}
We recall that the Schr\"odinger potentials \((f_{\mu,\nu}, g_{\mu,\nu})\) are characterized as the solutions of the Schr\"odinger fixed-point equation
\begin{equation}\label{eq:schrodinger-syst}
\begin{cases}
f_{\mu,\nu} = T_\varepsilon(g_{\mu,\nu}, \nu) \\
g_{\mu,\nu} = T_\varepsilon(f_{\mu,\nu}, \mu)
\end{cases}
\end{equation}
where the Sinkhorn mapping is defined by
\begin{equation*}
T_\varepsilon(f, \mu)(y) \coloneqq -\varepsilon\log\int\exp(\tfrac1\varepsilon(f(x)-\|x-y\|^2))\dd \mu(x).
\end{equation*}
\begin{proof}
We write \(\bar f, \bar g \coloneqq f_{\bar\mu,\bar\nu}, g_{\bar \mu, \bar \nu}\) for conciseness. By the Schr\"odinger system \eqref{eq:schrodinger-syst} we have
\begin{equation*}
\bar g(y-n) = -\varepsilon\log \int \exp(\tfrac1\varepsilon(\bar f(x-m) - \| y-x + m - n\|^2))\dd\mu(x)
\end{equation*}
and after expanding \(\|y-x + m-n\|^2\) and taking the quantities independent on \(x\) out of the \(\log\int\exp\), we obtain
\begin{align*}
\bar g (y-n) =& -\varepsilon\log \int \exp(\tfrac1\varepsilon(\bar f(x-m) + 2\langle m-n, x\rangle  - \|y-x\|^2))\dd\mu(x)\\
&+ 2\langle m-n, y\rangle +\|m-n\|^2
\end{align*}
which writes exactly as  \(g = T_\varepsilon(f, \mu)\) with \(f = \bar f(\cdot - m)+2\langle m-n, \cdot\rangle\) and \(g = \bar g(\cdot - n) + 2\langle n-m, \cdot\rangle - \| m-n\|^2\). Symmetrically we get \(T_\varepsilon(g, \nu) = f\) yielding the result.
\end{proof}
We can now prove that \cref{thm:Seps-gauss} holds for singular covariance matrices. 
\begin{proof}[Proof of \cref{thm:Seps-gauss}]
We take \(\Sigma_k, \Gamma_k \in \mathbb S_{++}^d\) for all \(k\) with \(\Sigma_k\to\Sigma\in\mathbb S^d_+\), and \(\Gamma_k\to \Gamma\in\mathbb S^d_+\). Denote \(\mu_k \coloneqq \mathcal N(0, \Sigma_k)\), \(\nu_k\coloneqq\mathcal N(0, \Gamma_k)\) (the uncentered case is recovered with \cref{lem:uncentered-potentials}), and \((f_k, g_k) \coloneqq (f_{\mu_k, \nu_k}, g_{\mu_k, \nu_k})\). By continuity of \(\Sigma', \Gamma' \mapsto F_{\Sigma'}^{\Gamma'}\) on \(\mathbb S^d_+\), we get that \(f_k\) (resp. \(g_k\)) converges uniformly on compact sets to \(f: x\mapsto \langle x, F_{\Sigma}^{\Gamma}x\rangle\) (resp. \(g: x\mapsto \langle x, F_\Gamma^\Sigma x\rangle\)), and for the weak-\textasteriskcentered{} convergence we have \(\mu_k\to\mu \coloneqq \mathcal N(0, \Sigma)\) and \(\nu_k\to\nu\coloneqq \mathcal N(0, \Gamma)\). Using the proof of \cite[Proposition 4, Step 2]{mena2019statistical}, we get that \((f,g) = (f_{\mu,\nu}, g_{\mu,\nu})\). It follows that the expression of the Sinkhorn divergence is also still valid, since \(\OT_\varepsilon(\mu_k, \nu_k) = \int f_k\dd\mu_k + \int g_k\dd\nu_k = \Tr(F_{\Sigma_k}^{\Gamma_k}\Sigma_k) + \Tr(F_{\Gamma_k}^{\Sigma_k}\Gamma_k) \to\Tr(F_{\Sigma}^{\Gamma}\Sigma) + \Tr(F_{\Gamma}^{\Sigma}\Gamma) = \int f\dd\mu+\int g\dd\nu = \OT_\varepsilon(\mu, \nu)\) and the expression given in \cref{thm:Seps-gauss} is continuous on \(\mathbb S^d_+\).
\end{proof}
We finish this section with a simple bound on the matrices involved in the Schr\"odinger potentials, which we will use recurrently in this paper.
\begin{lemma}\label{lem:Gbounded}
For \(\Sigma, \Gamma\in\mathbb S^d_+\), the eigenvalues of \(F_\Sigma^\Gamma\) belong to the interval \([1-\lambda_{\max}(\Gamma)/(2{\tilde\varepsilon}), 1]\) and those of \(F_\Sigma^\Sigma\) belong to \([0, 1]\). In particular, the eigenvalues of \(G_\Sigma^\Gamma\) are contained within \([-\lambda_{\mathrm{max}}(\Gamma)/\tilde\varepsilon, 2]\).
\end{lemma}
\begin{proof}
 Denoting \((\alpha_i)_i\) the eigenvalues of the matrix \(\Gamma^{\frac12}\Sigma\Gamma^{\frac12}\), the eigenvalues of \(((\Gamma^{\frac12}\Sigma\Gamma^{\frac12}+\tilde\varepsilon^2\id)^{\frac12}+\tilde\varepsilon\id)^{-1}\) are given by \(\frac{1}{\sqrt{\alpha_i+\tilde\varepsilon^2}+\tilde\varepsilon}\) which is between 0 and \(\frac{1}{2\tilde\varepsilon}\). It results that \(\frac{1}{2\tilde\varepsilon}\id - ((\Gamma^{\frac12}\Sigma\Gamma^{\frac12}+\tilde\varepsilon^2\id)^{\frac12}+\tilde\varepsilon\id)^{-1}\) is a positive semidefinite matrix, and so is the same matrix multiplied left and right by \(\Gamma^{\frac12}\), giving that the eigenvalues of \(\frac{1}{2\tilde\varepsilon}\Gamma - \Gamma^{\frac12}((\Gamma^{\frac12}\Sigma\Gamma^{\frac12}+\tilde\varepsilon^2\id)^{\frac12}+\tilde\varepsilon\id)^{-1}\Gamma^{\frac12}\) are non-negative, whence the fact that the eigenvalues of \(\Gamma^{\frac12}((\Gamma^{\frac12}\Sigma\Gamma^{\frac12}+\tilde\varepsilon^2\id)^{\frac12}+\tilde\varepsilon\id)^{-1}\Gamma^{\frac12}\) are bounded above by \(\frac{1}{2\tilde\varepsilon}\lambda_{\mathrm{max}}(\Gamma)\). They are also non-negative. Subtracting this from the identity matrix yields the bounds for \(F_\Sigma^\Gamma\). The eigenvalues of \(F_\Sigma^\Sigma = \id-\Sigma((\Sigma^2+\tilde\varepsilon^2\id)^{\frac12} + \tilde\varepsilon \id)^{-1}\) are given by \(1-\frac{\lambda_i}{\sqrt{\lambda_i^2+\tilde\varepsilon^2}+\tilde\varepsilon}\) for \((\lambda_i)_i\) the eigenvalues of \(\Sigma\), and are thus between 0 and 1. Finally, the bound on \(G_\Sigma^\Gamma\) follows by doing the difference of the two previous bounds and multiplying by 2.
\end{proof}

\section{Well-posedness of the flow} \label{sec:wellposed}
Throughout the rest of this paper, we fix \(\mu_\star \coloneqq \mathcal N(m_\star, \Sigma_\star)\) (with \(\Sigma_\star\) possibly singular unless stated otherwise) and study the PDE \eqref{eq:SWGF}. We will also denote \(G_\Sigma\) instead of \(G_\Sigma^{\Sigma_\star}\) to ease notations.

We first prove the existence of solutions by constraining the functional \(S_\varepsilon(\cdot, \mu_\star)\) to the set of Gaussian measures and obtaining its \(\lambda\)-convexity along generalized geodesics, yielding existence of gradient flows of this constrained functional which are proven to follow \eqref{eq:SWGF}. Then we show that the solutions of that PDE are also gradient flows of the unconstrained functional, and finally establish the uniqueness of solutions on a class of regular measures.
\subsection{Existence} \label{sec:wellposed:existence}
From now on, we denote \(\mathcal G \subset \mP_2(\bR^d) \) the set of (possibly singular) Gaussian distributions. We write \(\iota_{\mathcal G}\) the function defined on \(\mP_2(\bR^d)\) equal to \(0\) on \(\mathcal G\) and \(+\infty\) elsewhere, and we define
\begin{equation*}
E_{\mu_\star} \coloneqq S_\varepsilon(\cdot, \mu_\star) + \iota_{\mathcal G}.
\end{equation*}
We can now state our existence theorem, which also guarantees uniqueness on \(\mathcal G\).
\begin{theorem}\label{thm:µflow}
For any \(\mu_0\in\mathcal G\), there exists a unique Wasserstein gradient flow of \(E_{\mu_\star}\) starting at \(\mu_0\), which is described by the PDE \eqref{eq:SWGF} and is contained in \(\mathcal G\).
\end{theorem}
The existence and uniqueness are proven using \cref{thm:ggcvx} below, allowing to apply \cite[Theorem 11.2.1]{ambrosio2008gradient}. This result gives a Wasserstein gradient flow which follows the PDE \(\dot\mu_t = \div(\mu_t\partial_W^\circ E_{\mu_\star}(\mu_t))\) because of the minimal selection principle. The fact that it coincides with the expression in \eqref{eq:SWGF} is given by \cref{prop:d^0E}, and that the flow stays Gaussian comes from the fact that \(t\mapsto E_{\mu_\star}(\mu_t)\) decreases and starts at a finite value.
\begin{theorem}\label{thm:ggcvx}
Let \(\mu_\star \in \mathcal G\) and let \(\Sigma_\star\) be its covariance matrix. Then \(E_{\mu_\star}\) is \(-\lambda_{\mathrm{max}}(\Sigma_\star)/\tilde\varepsilon\)-convex along generalized geodesics.
\end{theorem}
In order to prove this, we first need to show that we can always pick a generalized geodesic which stays in \( \mathcal G\).
\begin{lemma}\label{lem:GaussianOTplanandglueing}
For \(\omega, \mu, \nu \in \mathcal G\), there exists a Gaussian distribution \(\bar \pi\in \mP_2((\bR^d)^3)\) with the three previous measures as marginals, such that \((\mathbf p_{1, 2})_\#\bar\pi \in \Pi_{\mathrm o}(\omega, \mu)\) and \((\mathbf p_{1, 3})_\#\bar\pi \in \Pi_{\mathrm o}(\omega, \nu)\). In particular, there is always a generalized geodesic \((\mu_t)_{t\in[0, 1]}\) between \(\mu\) and \(\nu\) of base \(\omega\) such that \(\mu_t \in\mathcal G\) for all \(t\).
\end{lemma}
\begin{proof}
For \(\pi_\mu \in \Pi_{\mathrm o}(\omega, \mu)\), the quantity \(\int \|x-y\|^2\dd\pi_\mu(x, y)\) depends only on the covariance matrix of \(\pi_\mu\) which has Gaussian marginals, so it can always be taken to be Gaussian. The same holds for \(\pi_\nu \in \Pi_{\mathrm o}(\omega, \nu)\). Denoting \(\dd \pi_\mu(x, y) = \dd \pi_{\mu|\omega}(y|x)\dd\omega(x)\) (resp.  \(\dd \pi_\nu(x, z) = \dd \pi_{\nu|\omega}(z|x)\dd\omega(x)\)) the disintegration of \(\pi_\mu\) (resp. \(\pi_\nu\)) with respect to \(\omega\), we build \(\dd\bar\pi(x, y, z) = \dd\pi_{\mu|\omega}(y|x) \dd\pi_{\nu|\omega}(z|x)\dd\omega(x)\) the usual gluing, and one can check that it is Gaussian by computing its characteristic function. The rest of the statement follows from the fact that for all \(t\), pushforwarding a Gaussian measure on \((\bR^d)^3\) by \((1-t) \mathbf p_{1, 2} + t\mathbf p_{1, 3}\) yields a Gaussian measure on \((\bR^d)^2\).
\end{proof}
With this lemma in hand, we can show our convexity result.
\begin{proof}[Proof of \cref{thm:ggcvx}.]
Without loss of generality, we consider the case of centered Gaussian measures. Indeed, the mean along a generalized geodesic is just the Euclidean geodesic between the source's and target's means, and \(\|\cdot-m_\star\|^2\) is 0-convex. Let \(\omega, \mu, \nu \in \mathcal G\) and let \(\bar\pi\) be a Gaussian plan with marginals \((\omega, \mu, \nu)\), satisfying \((\mathbf p_{1, 2})_\#\bar\pi \in \Pi_{\mathrm o}(\omega, \mu)\) and \((\mathbf p_{1, 3})_\#\bar\pi \in \Pi_{\mathrm o}(\omega, \nu)\), as given by \cref{lem:GaussianOTplanandglueing}. Then \(\mu_t \coloneqq ((1-t)\mathbf p_2+t\mathbf p_3)_\#\bar\pi\) is a Gaussian of covariance
\begin{equation*}
\Sigma_t = \int((1-t)y+tz)((1-t)y+tz)^T\dd\bar\pi(x, y, z)
\end{equation*}
and thus by differentiating twice under the integral,
\begin{equation*}
\ddot\Sigma_t = 2\int(y-z)(y-z)^T\dd\bar\pi(x, y, z)
\end{equation*}
which is a PSD matrix, also satisfying \(\Tr(\ddot \Sigma_t) = 2W^2_{\bar\pi}(\mu, \nu)\). Now observe that
\begin{equation*}
\frac{\dd^2}{\dd t^2}S_\varepsilon(\mu_t, \mu_\star) = \Tr(\ddot\Sigma_t \nabla_1 \mathcal B_\varepsilon(\Sigma_t, \Sigma_\star)) + \Tr(\dot \Sigma_t \nabla_1^2\mathcal B_\varepsilon(\Sigma_t, \Sigma_\star)(\dot \Sigma_t))
\end{equation*}
where the second term is non-negative by convexity (\cite[Proposition 6 (iii)]{janati2020entropic}), and the first one can be rewritten \(\Tr(\ddot\Sigma_t \nabla_1 \mathcal B_\varepsilon(\Sigma_t, \Sigma_\star)) = \frac12\Tr((\ddot\Sigma_t)^{\frac12} G_{\Sigma_t}(\ddot\Sigma_t)^{\frac12}) \geq -\frac{\lambda_{\mathrm{max}}(\Sigma_\star)}{2\tilde\varepsilon} \Tr(\ddot \Sigma_t)\) by \cref{lem:Gbounded}, and this last bound is equal to \(-\frac{\lambda_{\mathrm{max}}(\Sigma_\star)}{\tilde\varepsilon}W_{\bar\pi}^2(\mu, \nu)\). Therefore \(t\mapsto S_\varepsilon(\mu_t, \mu_\star)\) is \(-\frac{\lambda_{\mathrm{max}}(\Sigma_\star)}{\tilde\varepsilon}W_{\bar\pi}^2(\mu, \nu)\)-convex, yielding the result.
\end{proof}
Now that convexity and thus existence of a Wasserstein gradient flow of \(E_{\mu_\star}\) is established, we show that it indeed corresponds to the PDE \cref{eq:SWGF} by computing the minimal subdifferential of \(E_{\mu_\star}\).
\begin{proposition}\label{prop:d^0E}
Let \(\mu \in \mathcal G\), then the Wasserstein subdifferential of \(E_{\mu_\star}\) at \(\mu\) is given by \(\partial_W E_{\mu_\star}(\mu) = \nabla (f_{\mu, \mu_\star} - f_\mu) + \partial_W \iota_{\mathcal G}(\mu)\), and \(\partial_W^\circ E_{\mu_\star}(\mu) = \nabla (f_{\mu, \mu_\star} - f_\mu)\).
\end{proposition}
\begin{proof}
We denote \(\mu = \mathcal N(m, \Sigma)\) and \(\nu=\mathcal N(n, \Gamma)\). We first assume \(m = n = m_\star =  0\) and we shall deal with the uncentered case afterwards. Let \(\pi\in\Pi_{\mathrm o}(\mu,\nu)\) chosen to be Gaussian, and \(\mu_t \coloneqq ((1-t)\mathbf p_1+t\mathbf p_2)_\#\pi\) the corresponding geodesic, which is a Gaussian distribution of covariance \(\Sigma_t = \int ((1-t)x + ty)((1-t)x + ty)^T\dd \pi(x, y)\). Thus \(\frac{\dd}{\dd t}S_\varepsilon(\mu_t, \mu_\star) = \Tr(\dot\Sigma_t\tfrac12 G_{\Sigma_t})\) with \(\dot \Sigma_t = \int (((1-t)x + ty)(y-x)^T + (y-x)((1-t)x + ty)^T)d\pi(x, y)\), yielding
\begin{equation*}
\frac{\dd}{\dd t}S_\varepsilon(\mu_t, \mu_\star) = \int \langle G_{\Sigma_t}((1-t)x+ty), y-x\rangle \dd\pi(x,y).
\end{equation*}
Now consider that
\begin{align*}
\left|\frac{\dd}{\dd t} S_\varepsilon(\mu_t, \mu_\star)-\int \langle G_{\Sigma}x, y-x\rangle \dd\pi(x,y)\right| \leq & \int|\langle G_{\Sigma_t}((1-t)x+ty) - G_{\Sigma_t}x, y-x\rangle|\dd \pi(x, y)\\
&+\int|\langle G_{\Sigma_t}x - G_{\Sigma}x, y-x\rangle|\dd \pi(x, y).
\end{align*}
Using Cauchy-Schwarz and the definition of the operator norm, the first term is bounded by \(t \|G_{\Sigma_t}\|_{\mathrm{op}}W_2^2(\mu, \nu)\) and the second by \(\|G_{\Sigma_t} - G_{\Sigma}\|_{\mathrm{op}}\sqrt{\Tr(\Sigma)}W_2(\mu, \nu)\). From \cref{lem:Gbounded}, \(\|G_{\Sigma_t}\|_{\mathrm{op}}\) is bounded, and writing \(\Delta(\mu, \nu) \coloneqq S_\varepsilon(\nu, \mu_\star) - S_\varepsilon(\mu, \mu_\star) - \int \langle G_{\Sigma}x, y-x\rangle \dd \pi(x,y)\) we get
\begin{align*}
|\Delta(\mu, \nu)| &\leq \int_0^1\left|\frac{\dd}{\dd t}S_\varepsilon(\mu_t, \mu_\star) - \int \langle G_{\Sigma}x, y-x\rangle \dd \pi(x,y)\right|\dd t \\
&\leq C_1W_2^2(\mu, \nu) + C_2 W_2(\mu,\nu)\int_0^1\|G_{\Sigma_t} - G_{\Sigma}\|_{\mathrm{op}}\dd t
\end{align*}
where \(C_1, C_2\) are positive constants. The integral in the last term goes to 0 as \(\Gamma\to\Sigma\) by the dominated convergence theorem and the continuity of \(\Sigma' \mapsto G_{\Sigma'}\), whence \(\Delta(\mu, \nu) = o(W_2(\mu, \nu))\) and thus
\begin{equation*}
S_\varepsilon(\nu, \mu_\star) - S_\varepsilon(\mu, \mu_\star) = \int \langle G_{\Sigma}x, y-x\rangle \dd \pi(x,y) + o(W_2(\mu, \nu)).
\end{equation*}
We now extend the proof to uncentered measures. Let \(\bar\mu, \bar \nu, \bar \mu_\star\) be the centered counterparts of \(\mu, \nu, \mu_\star\). Recall that \(W_2^2(\mu,\nu) = \|m-n\|^2 + W_2^2(\bar\mu, \bar\nu)\), and that if \(\pi\in\Pi_o(\mu, \nu)\), then the plan \(\bar \pi \coloneqq (\mathbf p_1 - m, \mathbf p_2 - n)_\#\pi\) is optimal between \(\bar\mu\) and \(\bar\nu\). Applying the previous estimate to \(\bar \mu, \bar \nu, \bar \mu_\star\) with \(\bar \pi\), and using a Taylor expansion of \(m'\mapsto\|m'-m_\star\|^2\) at \(m\), we get
\begin{align*}
S_\varepsilon(\nu, \mu_\star) - S_\varepsilon(\mu, \mu_\star) =& \int \langle G_{\Sigma}x, y-x\rangle \dd \bar\pi(x,y) + o(W_2(\bar\mu, \bar\nu))  \\&+\langle 2(m-m_\star), n-m\rangle + o(\|m-n\|).
\end{align*}
By Cauchy-Schwarz we get \(o(W_2(\bar\mu, \bar\nu)) + o(\|m-n\|) = o(W_2(\mu, \nu))\), therefore 
\begin{align*}
S_\varepsilon(\nu, \mu_\star) - S_\varepsilon(\mu, \mu_\star) =& \int \langle G_{\Sigma}(x-m), y-n-x+m\rangle \dd \pi(x,y) \\
&+  \int\langle 2(m-m_\star), y-x\rangle \dd\pi(x, y)+ o(W_2(\mu, \nu))\\
=& \int \langle G_{\Sigma}(x-m)+2(m-m_\star), y- x\rangle \dd \pi(x,y)\\
&+ \int \langle G_{\Sigma}(x-m), m-n\rangle \dd \pi(x,y) + o(W_2(\mu, \nu))
\end{align*}
and since \(\int\langle G_{\Sigma}(x-m),m-n\rangle \dd \pi(x,y) = 0\), we get
\begin{equation}\label{eq:gradSeps}
S_\varepsilon(\nu, \mu_\star) - S_\varepsilon(\mu, \mu_\star) =\int \langle \nabla(f_{\mu, \mu_\star} - f_{\mu})(x), y- x\rangle \dd \pi(x,y) + o(W_2(\mu,\nu)).
\end{equation}

Next, we have that \(\xi \in \partial_W E_{\mu_\star}(\mu)\) if and only if for all \(\nu\in\mathcal G\), there exists an optimal transport plan \(\pi\) between \(\mu\) and \(\nu\) such that \(S_\varepsilon(\nu, \mu_\star) - S_\varepsilon(\mu, \mu_\star) \geq \int\langle\xi(x),y-x\rangle \dd \pi(x,y) + o(W_2(\mu, \nu))\). By substituting in \eqref{eq:gradSeps}, this last inequality is equivalent to
\begin{equation*}
0 \geq \int\langle\xi(x)-\nabla(f_{\mu, \mu_\star} - f_{\mu})(x),y-x\rangle \dd \pi(x,y) + o(W_2(\mu, \nu)),
\end{equation*}
i.e.~precisely \(\xi -\nabla(f_{\mu, \mu_\star} - f_{\mu}) \in \partial_W \iota_{\mathcal G}(\mu)\).
Now take \(\zeta \in \partial_W\iota_{\mathcal G}(\mu)\), \(v\in L^2_\mu(\bR^d)\) an affine vector field, and let \(\mu_t = (\id + tv)_\# \mu\). This defines an absolutely continuous curve such that \(\dot \mu_0 = -\div(\mu v)\). There exists \(\pi_t\) an optimal transport plan between \(\mu\) and \(\mu_t\) such that \(\int \langle\zeta(x), y-x\rangle \dd\pi_t(x, y) + o(W_2(\mu, \mu_t)) \leq 0\). We have \(W_2(\mu, \mu_t) = O(t)\) and \(\int \langle \zeta(x), \frac{y-x}{t}\rangle \dd\pi_t(x, y) \to \int \langle \zeta(x), v(x)\rangle \dd \mu(x)\) by \cite[Proposition 8.4.6]{ambrosio2008gradient}. Thus \[\int \left\langle \zeta(x), \frac{y-x}{t}\right\rangle \dd\pi_t(x, y) + o(1) = \int \langle \zeta(x), v(x)\rangle \dd \mu(x) + o(1) \leq 0\] i.e.~\(\langle \zeta, v\rangle_{L^2_\mu(\bR^d)}\leq 0\).
Using the same reasoning with \(-v\) yields \(\langle \zeta, v\rangle_{L^2_\mu(\bR^d)}= 0\). The minimal selection follows from \(\nabla(f_{\mu, \mu_\star} - f_{\mu})\) being affine, yielding \(\|\nabla(f_{\mu,\mu_\star} - f_\mu) + \zeta\|^2_{L^2_\mu(\bR^d)} = \|\nabla(f_{\mu,\mu_\star} - f_\mu)\|_{L^2_\mu(\bR^d)}^2 + \|\zeta\|_{L^2_\mu(\bR^d)}^2\) which is minimized for \(\zeta = 0\in\partial_W \iota_{\mathcal G}(\mu)\).
\end{proof}
\subsection{Characterization as the Wasserstein gradient flow of the Sinkhorn divergence}\label{sec:wellposed:WGFcharacterization}
Having proven the existence of solutions to \eqref{eq:SWGF}, we now make sure they are indeed Wasserstein gradient flows of the functional \(S_\varepsilon(\cdot, \mu_\star)\) in the general sense (so far we have obtained solutions as gradient flows of \(E_{\mu_\star}\) which is constrained to \(\mathcal G\)). To do so, the only thing to prove is that \( \nabla(f_{\mu,\mu_\star} - f_\mu) \) is a Wasserstein subdifferential of the unconstrained functional.
\begin{proposition}\label{prop:subdiff-Seps-G}
For any \(\mu\in\mathcal G\), \(\nabla(f_{\mu,\mu_\star} - f_{\mu})\) is a Wasserstein subdifferential of \(S_\varepsilon(\cdot, \mu_\star)\) at \(\mu\).
\end{proposition}
\begin{proof}
In the proof of \cite[Proposition 4]{janati2020debiased} it is shown that for any \(\nu\in\mathcal P_2(\bR^d)\), \(S_\varepsilon(\nu, \mu_\star) - S_\varepsilon(\mu, \mu_\star) \geq \int (f_{\mu, \mu_\star} - f_{\mu})\dd(\nu-\mu)\) (the assumption of subgaussianity made in that paper is not necessary for this inequality, which only relies on the suboptimality of well-chosen combinations of the cross and self transport potentials in the dual EOT problem). Taking an optimal transport plan \(\pi\in \Pi_{\mathrm o}(\mu, \nu)\) and writing \((\mu_t)_{t\in[0, 1]}\) the corresponding geodesic, the previous lower bound reads \(\int (f_{\mu, \mu_\star} - f_{\mu})\dd(\nu-\mu) = \int_0^1\left[\frac\dd{\dd t}\int (f_{\mu, \mu_\star} - f_{\mu})\dd\mu_t\right]\dd t\), where we can compute
\begin{equation*}
\frac\dd{\dd t}\int (f_{\mu, \mu_\star} - f_{\mu})\dd\mu_t = \int\langle\nabla(f_{\mu, \mu_\star} - f_{\mu})((1-t)x+ty), y - x\rangle \dd\pi(x, y).
\end{equation*}
Writing \(\mu = \mathcal N(m, \Sigma)\) and \(b \coloneqq -G_{\Sigma}m + 2(m - m_\star)\) so that \(\nabla(f_{\mu, \mu_\star} - f_{\mu})(x) = G_{\Sigma}x + b\) for conciseness, by integrating the previous equation with respect to \(t\) we get
\begin{equation*}
\int (f_{\mu, \mu_\star} - f_{\mu})\dd(\nu-\mu) = \int\langle G_{\Sigma}x+ b, y - x\rangle \dd\pi(x, y) + \frac12\int\langle G_{\Sigma}(y-x), y - x\rangle \dd\pi(x, y).
\end{equation*}
Using \cref{lem:Gbounded} and injecting in the first inequality of this proof, we get
\begin{equation*}
S_\varepsilon(\nu, \mu_\star) - S_\varepsilon(\mu, \mu_\star) \geq \int\langle \nabla(f_{\mu, \mu_\star} - f_{\mu})(x), y - x\rangle \dd\pi(x, y) - \frac{\lambda_{\mathrm{max}}(\Sigma_\star)}{2\tilde{\varepsilon}}W_2^2(\mu,\nu)
\end{equation*}
which allows us to conclude.
\end{proof}

\subsection{EVI and uniqueness for a class of regular measures} \label{sec:wellposed:EVI}
In \cref{sec:wellposed:existence}, we have used a \(\lambda\)-convexity argument which gave us uniqueness of the gradient flow, but only among Gaussian solutions since the functional was constrained to \(\mathcal G\). Considering that a similar convexity property is challenging to obtain on \(\mP_2(\bR^d)\), we prove uniqueness among curves contained in a subclass \(\mathcal R\) of regular measures, with strongly concave and smooth log densities as defined below. 
This class contains all non-singular Gaussian measures (but not singular ones, as opposed to \(\mathcal G\)), and is a subset of the subgaussian measures.
\begin{definition}\label{def:R}
We define \(\mathcal R\) as the class of probability measures \(\mu\in\mP_2(\bR^d)\) having a Lebesgue density proportional to \(\exp(-V)\) where \(V \colon \bR^d \to \bR \) is twice continuously differentiable and such that there exist \(\alpha_\mu, \beta_\mu >0\) such that \(\alpha_\mu \id \preceq \nabla^2V \preceq \beta_\mu \id\).
\end{definition}

We start with a result akin to \cref{prop:subdiff-Seps-G}, but on \(\mathcal R\) rather than \(\mathcal G\), which will allow us to state an EVI \eqref{eq:EVIdef} to obtain uniqueness on \(\mathbb R\).
\begin{proposition}\label{prop:subdiff-R}
For \(\Sigma_\star \in \mathbb S^d_{++}\) and \(\mu\in\mathcal R\), for all \(\nu\in\mathcal P_2(\bR^d)\) and all \(\pi \in \Pi_o(\mu, \nu)\), 
\begin{equation*}
S_\varepsilon(\nu, \mu_\star) - S_\varepsilon(\mu, \mu_\star) \geq \int\langle \nabla(f_{\mu, \mu_\star} - f_{\mu})(x), y - x\rangle \dd\pi(x, y) - \frac{\lambda_{\max}(\Sigma_\star)}{2\tilde\varepsilon}W_2^2(\mu,\nu).
\end{equation*}
\end{proposition}
\begin{proof}
For any \(\nu\in\mP_2(\bR^d)\), as in the previous proof we take \((\mu_t)_{t\in[0, 1]}\) a geodesic corresponding to a plan \(\pi \in \Pi_{\mathrm o}(\mu, \nu)\) and write
\begin{equation*}
S_\varepsilon(\nu, \mu_\star) - S_\varepsilon(\mu, \mu_\star) \geq  \int_0^1\int\langle\nabla (f_{\mu, \mu_\star}-f_{\mu})(x_t), y - x\rangle \dd\pi(x, y)\dd t,
\end{equation*}
where in the integral we write \(x_t = (1-t)x + t y\) for conciseness. We can now write
\begin{equation}\label{eq:Taylor}
\nabla (f_{\mu, \mu_\star}-f_{\mu})(x_t) = \nabla (f_{\mu, \mu_\star}-f_{\mu})(x) + \int_0^t\nabla^2 (f_{\mu, \mu_\star}-f_{\mu})(x_s)(y-x)\dd s
\end{equation}
since the Schr\"odinger potentials are \(\mathcal C^2\) for subgaussian measures (see \cite[Appendix A]{mena2019statistical}).
With the definition of \(\mathcal R\) we can apply \cite[Theorem 8]{chewi2023entropic}, yielding that
\begin{equation*}
\nabla^2(f_{\mu,\mu_\star}-f_{\mu}) \succeq \frac\varepsilon2\left[\alpha_{\mu}\left(\sqrt{1+\frac{16}{\varepsilon^2\beta_\mu\alpha_\mu}}-1\right) + \beta_{\mu}\left(1-\sqrt{1+\frac{16}{\varepsilon^2\beta_\mu\alpha_{\mu_\star}}}\right)\right] \id,
\end{equation*}
where \(\alpha_\mu, \beta_\mu\) are as in \cref{def:R}, and \(\alpha_{\mu_\star} = \lambda_{\min}(\Sigma_\star^{-1}) = \frac{1}{\lambda_{\max}(\Sigma_\star)}\). Note that the first term of the sum between square brackets is non-negative, and the second term can be bounded below by \(-\frac{8}{\varepsilon^2\alpha_{\mu_\star}}\) (the function \(x\mapsto x(1-\sqrt{1+\frac ax})\) is decreasing and converges to \(-\frac a2\) as \(x\to+\infty\)). We thus get \(\nabla^2(f_{\mu,\mu_\star}-f_{\mu}) \succeq-\frac{\lambda_{\mathrm{max}}(\Sigma_\star)}{\tilde\varepsilon}\id\), so that when taking the scalar product of \eqref{eq:Taylor} with \(y-x\) we obtain
\begin{equation*}
\langle\nabla (f_{\mu, \mu_\star}-f_{\mu})(x_t), y - x\rangle \geq \langle\nabla (f_{\mu, \mu_\star} - f_{\mu})(x), y - x\rangle - \frac{\lambda_{\mathrm{max}}(\Sigma_\star)}{\tilde\varepsilon} t\| y-x\|^2.
\end{equation*}
Injecting in the first inequality of this proof, we get
\begin{equation*}
S_\varepsilon(\nu, \mu_\star) - S_\varepsilon(\mu, \mu_\star) \geq \int\langle \nabla (f_{\mu, \mu_\star}-f_{\mu})(x), y - x\rangle \dd\pi(x, y) - \frac{\lambda_{\mathrm{max}}(\Sigma_\star)}{2\tilde\varepsilon}W_2^2(\mu,\nu),
\end{equation*}
proving the claim.
\end{proof}
We are now ready to state the EVI and uniqueness.
\begin{theorem}[EVI on a regular class of measures]\label{thm:EVI}
Let \((\mu_t)_t\) be an absolutely continuous curve in \((\mathcal R, W_2)\) solution of \eqref{eq:SWGF}. Then for all \(\nu \in \mP_2(\bR^d)\), it holds
\begin{equation*}
\frac{\dd}{\dd t}\frac12 W_2^2(\mu_t, \nu) \leq S_\varepsilon(\nu, \mu_\star) - S_\varepsilon(\mu_t, \mu_\star) + \frac{\lambda_{\max}(\Sigma^\star)}{2\tilde\varepsilon}W_2^2(\mu_t, \nu).
\end{equation*}
In particular, fixing a non-singular Gaussian initial datum \(\mu_0\in\mathcal G\cap\mathcal R\), such a curve \((\mu_t)_t\) is unique, stays in \(\mathcal G\) and is a Wasserstein gradient flow of \(S_\varepsilon(\cdot, \mu_\star)\).
\end{theorem}
\begin{proof}
We have by \cite[Theorem 8.4.7]{ambrosio2008gradient} that
\begin{equation*}
\frac{\dd}{\dd t}\frac12W_2^2(\mu_t, \nu) = \int \langle\nabla(f_{\mu_t,\mu_\star}-f_{\mu_t})(x), y-x\rangle \dd \pi_t^\nu(x,y)
\end{equation*}
for some \(\pi_t^\nu \in \Pi_{\mathrm o}(\mu_t,\nu)\). By injecting the inequality given by \cref{prop:subdiff-R}, we get the EVI, and the uniqueness follows (as explained in \cref{sec:background:WGFs}).
\end{proof}

\begin{remark}[Difficulties with a general uniqueness result]
In our proof, as in the proof of \(\lambda\)-convexity of \(S_\varepsilon\) in the case of compactly supported measures \cite[Theorem 4.1]{carlier2024displacement}, a key ingredient is the regularity of the Schr\"odinger potentials. To prove uniqueness of the flow on the entirety of \(\mP_2(\bR^d)\), one would most likely need to obtain similar smoothness results in this more general case, and it is not yet known whether this is possible. Similarly, deriving a stability result for the Schr\"odinger map \( (\mu,\nu) \mapsto (f_{\mu,\nu}, g_{\mu,\nu}) \), as done in \cite{carlier2024displacement} in the compact setting, would be a powerful tool to control \( \nabla(f_{\mu_t,\mu_\star}-f_{\mu_t}) \) and hence prove that  \( (\mu_t)_t \)  stays in $\mathcal{R}$ (and thus in $\mathcal{G}$) whenever $\mu_0 \in \mathcal{G}$.
\end{remark}

\section{Flow in the Bures--Wasserstein space and convergence} \label{sec:equation-cv}
From the results of \cref{sec:wellposed}, we can now focus our analysis on the solution of \eqref{eq:SWGF} that stays in the space of Gaussian distributions. 
Using \cref{sec:background:preliminary-gaussian}, we write the evolution equation of the corresponding means and covariance matrices, and use this expression to obtain convergence results.
\subsection{General results} \label{sec:equation-cv:general}
We start this section with our main result, writing the mean and covariance evolutions and giving a sufficient condition for convergence (the non-singularity of the initial measure) and characterizing precisely the limit points in general. 
\begin{theorem}\label{thm:covflow}
The unique Gaussian solution of \eqref{eq:SWGF} starting at \(\mu_0 = \mathcal N(m_0, \Sigma_0)\) is given by \(\mu_t = \mathcal N(m_t, \Sigma_t)\) with \(m_t = m_\star + e^{-2t}(m_0-m_\star)\) and \(\Sigma_t\) solution of
\begin{equation}\label{eq:Sigmaflow}
\dot\Sigma_t = -(G_{\Sigma_t}\Sigma_t + \Sigma_tG_{\Sigma_t})
\end{equation}
initialized at \(\Sigma_0\). Additionally, we have:
\begin{thmenum}
\item If \(\Sigma_0\) is invertible, so is \(\Sigma_t\) for a.e.~\(t\), and conversely if \(\Sigma_0\) is singular then so is \(\Sigma_t\) for a.e.~\(t\). \label{thm:covflow:inv}
\item The curve \((\Sigma_t)_t\) is bounded, and as \(t\to+\infty\), \(\Sigma_t\to P\diag((\lambda_i)_i)P^T\) where \(\Sigma_\star = P\diag((\lambda_i^\star)_i)P^T\) and \(\lambda_i \in \{0, \lambda_i^\star\}\) for all \(i\). If \(\Sigma_0\) is invertible, then \(\Sigma_t \to \Sigma_\star\). \label{thm:covflow:cv}
\item The energy dissipation equality \(\frac{\dd}{\dd t} \mathcal B_\varepsilon(\Sigma_t, \Sigma_\star) = - \Tr(G_{\Sigma_t}\Sigma_tG_{\Sigma_t})\) holds for a.e.~\(t\). \label{thm:covflow:ede}
\end{thmenum} 
\end{theorem}
Note that this theorem gives \refthmitem{thm:cv-criterion}{ns} and \ref{thm:cv-criterion:s2ns}, since for Gaussian distributions, convergence of the means and of the covariances is enough to get convergence of the measures for \(W_2\), and if \(\Sigma_0\) is singular, then \(\Sigma_t\) cannot converge to a non-singular matrix.
\begin{remark}[Link with the Bures--Wasserstein geometry]
The space of covariance matrices endowed with the Bures--Wasserstein distance is a Riemannian manifold \cite{takatsu2011wasserstein, bhatia2019bures}. The Bures--Wasserstein gradient of a functional \(E\) can be computed as \(\nabla_{BW}E(\Sigma) = 2(\nabla E(\Sigma)\Sigma + \Sigma \nabla E(\Sigma))\), where we distinguish the Bures--Wasserstein gradient \(\nabla_{BW}\) from the Euclidean gradient \(\nabla\) (see \cref{sec:BWGFs}). Thus, since \(\nabla_1\mathcal B(\Sigma, \Sigma_\star) = \frac12G_\Sigma\) as given by \cref{thm:Seps-gauss}, the evolution \eqref{eq:Sigmaflow} indeed corresponds to the Bures--Wasserstein gradient flow of \(E = \mathcal B_\varepsilon(\cdot, \Sigma_\star)\), i.e.~\(\dot\Sigma_t = -\nabla_{BW}E(\Sigma_t)\). We will use this fact to prove the convergence of the flow thanks to \cref{lem:BW-flow-cv}. Note also that the energy dissipation equality \ref{thm:covflow:ede} reads \(\frac{\dd}{\dd t}E(\Sigma_t) = -\mathbf g_{\Sigma_t}(\dot\Sigma_t, \dot\Sigma_t)\), where \(\mathbf g\) is the Bures--Wasserstein metric tensor as defined in \cref{sec:BWGFs}.
\end{remark}

First, we prove the expression of the flow and its properties at finite \(t\), and we postpone the proof regarding the convergence as we will need extra lemmas.
\begin{proof}[Proof of \eqref{eq:Sigmaflow}, \ref{thm:covflow:inv} and \ref{thm:covflow:ede}.]
We write for conciseness \(G_t = G_{\Sigma_t}\) and \(v_t = \nabla (f_{\mu_t,\mu\star}-f_{\mu_t})\). Then we have \(\dot m_t = \frac{\dd}{\dd t}\int x \dd \mu_t(x) = -\int v_t(x)\dd\mu_t(x)\), and thus \(\dot m_t = -G_t(m_t-m_t) - 2 (m_t - m_\star) =2 (m_\star - m_t)\). The expression on \(m_t\) follows. Note that we can now write \(v_t = G_t(x-m_t) - \dot m_t\). For \(\Sigma_t\), we write \(M_t = \int xx^Td\mu_t(x) = \Sigma_t + m_tm_t^T\) and we have \(\dot M_t = -\int (xv_t(x)^T + v_t(x)x^T)\dd\mu_t(x)\). Then compute \(\int v_t(x)x^T\dd\mu_t(x) = \int (G_t(x-m_t)-\dot m_t)x^Td\mu_t(x) = G_t (M_t -m_tm_t^T) - \dot m_t m_t^T\). Hence,
\begin{equation*}
\dot M_t = -(G_t(M_t -m_tm_t^T)  + (M_t -m_tm_t^T)G_t) + \dot m_t m_t^T + m_t\dot m_t^T.
\end{equation*}
Using \(\Sigma_t = M_t - m_tm_t^T\) and \(\dot \Sigma_t = \dot M_t - \dot m_tm_t^T - m_t\dot m_t^T\), we deduce \eqref{eq:Sigmaflow}.

It follows that \(\frac{\dd}{\dd t}\det(\Sigma_t) = \Tr(\dot\Sigma_t\mathrm{adj}(\Sigma_t)) = - 2\Tr(G_{\Sigma_t})\det(\Sigma_t)\) (where adj denotes the adjugate matrix), giving \(\det(\Sigma_t) = \exp\left(-2\int_0^t\Tr(G_{\Sigma_s})\dd s\right)\det(\Sigma_0)\). Since \(s\mapsto\Tr(G_{\Sigma_s})\) is bounded (\cref{lem:Gbounded}), we get \ref{thm:covflow:inv}. Finally, applying the chain rule, \(\frac{\dd}{\dd t}\mathcal B_\varepsilon(\Sigma_t, \Sigma_\star) = \Tr(\dot\Sigma_t\tfrac12G_{\Sigma_t}) =- \Tr(G_{\Sigma_t}\Sigma_tG_{\Sigma_t})\) i.e.~\ref{thm:covflow:ede}. Note that this could also be obtained from \cite[Theorem~11.2.1, Eq.~(11.2.4)]{ambrosio2008gradient}, using the convexity of \( S_\varepsilon(\cdot,\mu_\star) \) on the space of Gaussian measures given by \cref{thm:ggcvx}.
\end{proof}
As the flow on the means has an explicit expression that converges exponentially to the target mean, and is independent from the flow on the covariances, we now assume all measures to be centered and focus on the convergence of the covariance. To obtain the result, we use Lemmas \ref{lem:cvtocrit} and \ref{lem:critpoints} below, respectively giving convergence to a critical point and characterizing such points.
\begin{lemma}\label{lem:cvtocrit}
Let \((\Sigma_t)_t\) follow the flow \eqref{eq:Sigmaflow}. Then it is bounded and converges to some \(\Sigma_\infty\), which must be a critical point in the sense that \(G_{\Sigma_\infty}\Sigma_\infty G_{\Sigma_\infty} = 0\).
\end{lemma}
\begin{proof}
That \((\Sigma_t)_t\) is bounded comes from the fact that the sublevel sets of \(\mathcal B_\varepsilon(\cdot, \Sigma_\star)\) are bounded and that this functional decreases over the flow by \refthmitem{thm:covflow}{ede}. To see that the sublevel sets are bounded, writing the expression of \(\mathcal B_\varepsilon\) given in \cref{thm:Seps-gauss} in terms of eigenvalues, we have
\begin{equation}\label{eq:Beps-eig}
\begin{aligned}
\mathcal B_\varepsilon(\Sigma, \Sigma_\star) =& \tfrac\varepsilon2\sum_i \left[\log(\sqrt{4\alpha_i+\tfrac{\varepsilon^2}{4}}+\tfrac\varepsilon2) - \tfrac12\log(\sqrt{4\lambda_i^2+\tfrac{\varepsilon^2}{4}}+\tfrac\varepsilon2)\right] \\&+ \sum_i\left[ \tfrac12\sqrt{4\lambda_i^2+\tfrac{\varepsilon^2}{4}} - \sqrt{4\alpha_i+\tfrac{\varepsilon^2}{4}}\right] - \tfrac{1}{2}B_\varepsilon(\Sigma_\star, \Sigma_\star)
\end{aligned}
\end{equation}
where \((\alpha_i)_i\) are the eigenvalues of \(\Sigma^{\frac12}_\star\Sigma\Sigma_\star^{\frac12}\) and \((\lambda_i)_i\) those of \(\Sigma\). Considering that \(\alpha_i \leq \lambda_{\max}(\Sigma)\lambda_{\max}(\Sigma_\star)\) for all \(i\), taking \(\lambda_{\max}(\Sigma)\to+\infty\) we get that the leading term is equivalent to \(\lambda_{\max}(\Sigma)\), making \(\mathcal B_\varepsilon(\Sigma, \Sigma_\star)\) also diverge to \(+\infty\), yielding bounded sublevel sets.
Seeing that this function is analytic, one can use Lojasiewicz-type arguments to obtain convergence to a limit, see \cref{lem:BW-flow-cv}. Additionally, using \cref{lem:Gbounded}, \((\dot\Sigma_t)_t\) is bounded and thus \(t\mapsto\Sigma_t\) is uniformly continuous, therefore \(t\mapsto\Tr(G_{\Sigma_t}\Sigma_tG_{\Sigma_t})\) is uniformly continuous. In addition, for all \(t\) we have \(\int_0^{t} \Tr(G_{\Sigma_s}\Sigma_sG_{\Sigma_s})\dd s = \mathcal B_\varepsilon(\Sigma_0, \Sigma_\star) - \mathcal B_\varepsilon(\Sigma_t, \Sigma_\star) \leq \mathcal B_\varepsilon(\Sigma_0, \Sigma_\star)\) by \refthmitem{thm:covflow}{ede}, yielding \(\int_0^{+\infty} \Tr(G_{\Sigma_s}\Sigma_sG_{\Sigma_s})\dd s<+\infty\). It follows that \(\Tr(G_{\Sigma_t}\Sigma_tG_{\Sigma_t})\to 0\) as \(t\to+\infty\). Thus, accumulation points must satisfy \(\Tr(G_{\Sigma_\infty}\Sigma_\infty G_{\Sigma_\infty})=0\), which implies that the matrix in the trace is null since it is positive semidefinite.
\end{proof}
\begin{lemma}\label{lem:critpoints}
A covariance \(\Sigma\) is critical in the sense of \cref{lem:cvtocrit} if and only if it is of the form \(P\diag((\lambda_i)_i) P^T\), where \(P\) is such that \(\Sigma_\star = P\diag((\lambda_i^\star)_i) P^T\) and \(\lambda_i \in \{0, \lambda_i^\star\}\).
\end{lemma}
\begin{proof}
If \(G_\Sigma \Sigma G_\Sigma = (G_\Sigma \Sigma^{\frac12})(G_\Sigma\Sigma^{\frac12})^T = 0\), we have \(G_\Sigma\Sigma^{\frac12} = \Sigma^{\frac12}G_\Sigma  = 0\) and thus

\begin{equation}\label{eq:GS=SG=0}
    G_\Sigma\Sigma = \Sigma G_\Sigma = 0.
\end{equation}

For the sake of conciseness, for \(X\in\mathbb S_+^d\) we denote \(J(X) = 2((X+\tilde{\varepsilon}^2\id)^{\frac12}+\tilde\varepsilon\id)^{-1}\), this way \(G_\Sigma = \Sigma J(\Sigma^2) - \Sigma_\star^{\frac12}J(\Sigma_\star^{\frac12}\Sigma\Sigma_\star^{\frac12})\Sigma_\star^\frac12 \) and thus \eqref{eq:GS=SG=0} writes

\begin{equation}\label{eq:starsig=sigstar}
\Sigma_\star^{\frac12}J(\Sigma_\star^{\frac12}\Sigma\Sigma_\star^{\frac12})\Sigma_\star^\frac12\Sigma = \Sigma \Sigma_\star^{\frac12}J(\Sigma_\star^{\frac12}\Sigma\Sigma_\star^{\frac12})\Sigma_\star^\frac12 = \Sigma^2J(\Sigma^2).
\end{equation}

Since we always have \(J(X)X = XJ(X)\), taking \(X = \Sigma_\star^{\frac12}\Sigma\Sigma_\star^{\frac12}\) and multiplying by \(\Sigma_\star^{\frac12}\) on both sides, we get
\begin{equation*}
\Sigma_\star^{\frac12} J(\Sigma_\star^{\frac12}\Sigma\Sigma_\star^{\frac12})\Sigma_\star^{\frac12}\Sigma\Sigma_\star=\Sigma_\star\Sigma\Sigma_\star^{\frac12}J(\Sigma_\star^{\frac12}\Sigma\Sigma_\star^{\frac12})\Sigma_\star^{\frac12}.
\end{equation*}
Injecting \eqref{eq:starsig=sigstar} in the above equation, we get \(\Sigma^2J(\Sigma^2)\Sigma_\star = \Sigma_\star \Sigma^2J(\Sigma^2)\), which implies that \(\Sigma_\star\) and \(\Sigma^2J(\Sigma^2)\) can be codiagonalized, and since the latter is diagonalizable exactly in the same bases as \(\Sigma\), we deduce that \(\Sigma_\star\) and \(\Sigma\) are codiagonalisable. We thus have an orthonormal \(P\) such that \(\Sigma = P\diag((\lambda_i)_i) P^T\), \(\Sigma_\star = P\diag((\lambda^\star_i)_i) P^T \) and therefore \(G_\Sigma\Sigma = P\diag((\ell_i)_i)P^T\) with
\begin{equation*}
\ell_i = 2\lambda_i\left(\frac{\lambda_i}{\sqrt{\lambda_i^2 + \tilde\varepsilon^2}+\tilde\varepsilon}- \frac{\lambda^\star_i}{\sqrt{\lambda^\star_i\lambda_i + \tilde\varepsilon^2}+\tilde\varepsilon}\right).
\end{equation*}
Since \(\ell_i =0\) for all \(i\), either \(\lambda_i = 0\) or \(\left[\lambda^\star_i \sqrt{\lambda_i^2 + \tilde\varepsilon^2} - \lambda_i\sqrt{\lambda^\star_i\lambda_i + \tilde\varepsilon^2}\right]  + \tilde\varepsilon(\lambda_i^\star - \lambda_i)= 0\), and since the two terms in this sum have the same sign, this last condition is equivalent to \(\lambda_i = \lambda_i^\star\).

Conversely, if \(\Sigma\) is of the given form then the above computation shows \(G_\Sigma\Sigma = 0\) and thus that \(\Sigma\) is critical.
\end{proof}
We can now prove the asymptotic behavior announced in \cref{thm:covflow}.
\begin{proof}[Proof of \refthmitem{thm:covflow}{cv}]
Applying \cref{lem:cvtocrit} and \cref{lem:critpoints} yields the first part of the statement. Thus we have \(\Sigma_t\to \Sigma_\infty = P\diag((\lambda_i)_i)P^T\) with \(\lambda_i\in\{0,\lambda_i^\star\}\). Taking \(\Sigma_0\) to be invertible, we look for a contradiction by assuming that the set \(\mathcal I\) of indices \(i\) such that \(\lambda_i = 0\) and \(\lambda_i^\star \neq 0\) is non-empty. By the continuity of \(\Sigma \mapsto G_\Sigma\), we have that \(G_{\Sigma_t}\) converges to \(P\diag((\ell_i)_i)P^T\) with
\begin{align*}
\ell_i &= \frac{\lambda_i}{\sqrt{\lambda_i^2+\tilde\varepsilon^2}+\tilde\varepsilon} - \frac{\lambda_i^\star}{\sqrt{\lambda_i\lambda_i^\star+\tilde\varepsilon^2}+\tilde\varepsilon}\\
&= \begin{cases}
0 \text{ if } \lambda_i = \lambda_i^\star,\\
- \frac{\lambda_i^\star}{2\tilde \varepsilon} \text{ if } \lambda_i = 0 \neq \lambda_i^\star.
\end{cases}
\end{align*}
By continuity of the trace, we therefore have
\begin{equation}\label{eq:TrGtoneg}
\Tr(G_{\Sigma_t}) \to \sum_{i\in\mathcal I} - \frac{\lambda_i^\star}{2\tilde \varepsilon} < 0.
\end{equation}
Using \refthmitem{thm:covflow}{inv} we can write for a.e.~\(t\) that \(\frac{\dd}{\dd t}\log\det(\Sigma_t) = -2\Tr(G_{\Sigma_t})\), which is positive for \(t\) large enough because of \eqref{eq:TrGtoneg}. It results that \((\log\det(\Sigma_t))_t\) is lower bounded, so by continuity \(\log\det(\Sigma_\infty) > -\infty\), which contradicts the fact that there are \(i\) such that \(\lambda_i = 0\). 
\end{proof}

\subsection{Study of the case where \texorpdfstring{\(\Sigma_0\) and \(\Sigma_\star\)}{Σ0 and Σ*} commute} \label{sec:equation-cv:commute}
In this section, we assume that \(\Sigma_0\) and \(\Sigma_\star\) commute and can thus be diagonalized in the same orthonormal basis \(P\), i.e.~\(\Sigma_0 = P\diag((\lambda_i^0)_i)P^T\) and \(\Sigma_\star = P\diag((\lambda_i^\star)_i)P^T\). That assumption reduces the evolution \eqref{eq:Sigmaflow} to \(d\) evolutions in \(\bR\) (for the \(d\) eigenvalues), which allows us to obtain convergence rates. This is summarized in the proposition below.
\begin{proposition}\label{prop:eigflow}
Under the assumption above, the evolution \eqref{eq:Sigmaflow} can be written \(\Sigma_t = P\diag((\lambda_i(t))_i)P^T\), where for every \(i\), \(\lambda_i(0)=\lambda_i^0\) and for a.e.~\(t\) holds
\begin{equation}\label{eq:eigflow}
\frac{\dd}{\dd t}\lambda_i(t) = 4\lambda_i(t)\left(\frac{\lambda_i^\star}{\sqrt{\lambda_i^\star\lambda_i(t)+\tilde\varepsilon^2}+\tilde\varepsilon} - \frac{\lambda_i(t)}{\sqrt{\lambda_i(t)^2+\tilde\varepsilon^2}+\tilde\varepsilon}\right).
\end{equation}
Additionally, we have the following convergence rates for the eigenvalues:
\begin{thmenum}
\item If \(\lambda^\star_i > 0\), then \(|\lambda_i(t)-\lambda_i^\star| \leq e^{-C_i^at}|\lambda_i^0 - \lambda_i^\star|\) \label{prop:eigflow:l*>0}
\item If \(\lambda_i^\star = 0\), then \(|\lambda_i(t)-\lambda_i^\star| = \lambda_i(t) \leq \frac{\lambda_i^0}{1+\lambda_i^0C_i^bt}\)\label{prop:eigflow:l*=0}
\end{thmenum}
where
\begin{align*}
C_i^a &= \frac{4\tilde\varepsilon\min(\lambda_i^\star, \lambda_i^0)}{(\sqrt{\lambda_i^\star\max(\lambda_i^0, \lambda_i^\star) + \tilde\varepsilon^2} + \tilde\varepsilon)(\sqrt{\max(\lambda_i^0, \lambda_i^\star)^2+\tilde\varepsilon^2}+\tilde\varepsilon)},\\
C_i^b &= \frac{4}{\sqrt{(\lambda_i^0)^2+\tilde\varepsilon^2}+\tilde\varepsilon}.
\end{align*}
This translates to the following estimate on the functional:
\begin{equation}\label{eq:Seps-cv-rate}
S_\varepsilon(\mu_t, \mu_\star) \leq \sum_{\lambda_i^\star >0}L_ie^{-C_i^a t}|\lambda_i^0-\lambda_i^\star| + \sum_{\lambda_i^\star  = 0}L_i\frac{\lambda_i^0}{1+\lambda_i^0C_i^bt}
\end{equation}
with \(L_i = 2(1 + \frac{\lambda_i^\star}{\tilde\varepsilon})\).
\end{proposition}
Note that in the evolution \eqref{eq:eigflow}, if \(\lambda_i^0 = 0\) then \(\lambda_i(t) = 0\) for all \(t\) (and \(C_i^a =0\)). Convergence to the target is therefore only possible if for all \(i\) such that \(\lambda_i^0 = 0\), we also have \(\lambda_i^\star = 0\), i.e.~if \(\ker(\Sigma_0)\subseteq\ker(\Sigma_\star)\) or equivalently that \(\supp(\mu_\star) \subseteq \supp(\mu_0)\). In that case, the estimate \eqref{eq:Seps-cv-rate} shows that convergence is exponential whenever the second term is null, i.e.~when \(\lambda_i^\star = 0\) implies \(\lambda_i^0 = 0\) or equivalently \(\supp(\mu_0) \subseteq \supp(\mu_\star)\), and sublinear otherwise. This is precisely what we announced in \refthmitem{thm:cv-criterion}{commute}.
\begin{proof}
One can differentiate \(\lambda \mapsto  4\lambda\left(\frac{\lambda_i^\star}{\sqrt{\lambda_i^\star\lambda+\tilde\varepsilon^2}+\tilde\varepsilon} - \frac{\lambda}{\sqrt{\lambda^2+\tilde\varepsilon^2}+\tilde\varepsilon}\right)\) and see that its derivative is bounded on \(\bR_+\), giving local existence of solutions to \eqref{eq:eigflow} by Cauchy-Lipschitz. One can also show that if \(\lambda_i(t) \geq \lambda_i^\star\) then \(\frac{\dd}{\dd t}\lambda_i(t) \leq 0\) and vice versa, meaning a solution is monotone and bounded between \(\lambda_i^0\) and \(\lambda_i^\star\), allowing to extend it to all non-negative \(t\). Then observe that the right-hand side of \eqref{eq:eigflow} describes exactly the diagonal coefficients of \(-(G_{\Sigma_t}\Sigma_t + \Sigma_tG_{\Sigma_t}) = -2G_{\Sigma_t}\Sigma_t\) in basis \(P\), yielding the solution by uniqueness of the flow. For the convergence rate \ref{prop:eigflow:l*>0}, if \(\lambda_i^0 \geq \lambda_i^\star\) and thus \(\lambda_i(t) \geq \lambda_i^\star\) for all \(t\), we write
\begin{equation*}
\frac{\dd}{\dd t}(\lambda_i(t)-\lambda_i^\star) = 4\lambda_i(t)\frac{\lambda_i^\star\sqrt{\lambda_i(t)^2+\tilde\varepsilon^2}-\lambda_i(t)\sqrt{\lambda_i^\star\lambda_i(t)+\tilde\varepsilon^2} + \tilde\varepsilon(\lambda_i^\star - \lambda_i(t))}{(\sqrt{\lambda_i^\star\lambda_i(t)+\tilde\varepsilon^2}+\tilde\varepsilon)(\sqrt{\lambda_i(t)^2+\tilde\varepsilon^2}+\tilde\varepsilon)}.
\end{equation*}
The numerator is bounded above by \(\tilde\varepsilon(\lambda_i^\star - \lambda_i(t))\) (which is non-positive), \(\lambda_i(t) \geq \lambda_i^\star\), and the denominator is bounded above by \((\sqrt{\lambda_i^\star\lambda_i^0+\tilde\varepsilon^2}+\tilde\varepsilon)(\sqrt{(\lambda_i^0)^2+\tilde\varepsilon^2}+\tilde\varepsilon)\). By Gronwall's lemma, we get \(\lambda_i(t)-\lambda_i^\star\leq e^{-C_it}(\lambda_i^0-\lambda_i^\star)\) with \(C_i = 4\frac{\tilde\varepsilon\lambda_i^\star}{(\sqrt{\lambda_i^\star\lambda_i^0+\tilde\varepsilon^2}+\tilde\varepsilon)(\sqrt{(\lambda_i^0)^2+\tilde\varepsilon^2}+\tilde\varepsilon)}\). With a similar reasoning when \(\lambda_i^0 \leq \lambda_i^\star\), we get \(\lambda_i^\star-\lambda_i(t)\leq e^{-C'_it}(\lambda_i^\star-\lambda_i^0)\) with \(C'_i = 4\frac{\tilde\varepsilon\lambda_i^0}{(\sqrt{(\lambda_i^\star)^2+\tilde\varepsilon^2}+\tilde\varepsilon)(\sqrt{(\lambda_i^\star)^2+\tilde\varepsilon^2}+\tilde\varepsilon)}\). Combining these two cases yields the estimate.
For \ref{prop:eigflow:l*=0}, with \(\lambda_i^\star=0\) the equation \eqref{eq:eigflow} simplifies to
\begin{equation*}
  \frac{\dd}{\dd t}\lambda_i(t) = - 4 \frac{\lambda_i(t)^2}{\sqrt{\lambda_i(t)^2+\tilde\varepsilon^2}+\tilde\varepsilon}
\end{equation*}
which gives that the flow decreases and thus we get \(\frac{\dd}{\dd t}\lambda_i(t) \leq -C_i^b\lambda_i(t)^2\). Writing \(u(t) = \frac{\lambda_i^0}{1+\lambda_i^0C_i^bt}\) the solution to \(u(0) = \lambda_i^0\), \(\frac{\dd}{\dd t}u(t) = -C_i^bu(t)^2\), we get \(\frac{\dd}{\dd t}(\lambda_i(t) - u_t)\leq -C_i^b(\lambda_i(t)+u(t))(\lambda_i(t)-u(t))\) and applying Gronwall's lemma gives \(\lambda_i(t) - u(t) \leq \exp(-C_i^b\int_0^t(\lambda_i(s)+u(s))\dd s)(\lambda_i(0)-u(0)) = 0\) i.e.~\ref{prop:eigflow:l*=0} holds.

For the estimate on the functional, we rewrite \eqref{eq:Beps-eig} under our commuting assumption, yielding
\begin{align*}
\mathcal B_\varepsilon(\Sigma, \Sigma_\star) =& \tfrac\varepsilon2\sum_i \left[\log(\sqrt{4\lambda_i^\star\lambda_i+\tfrac{\varepsilon^2}{4}}+\tfrac\varepsilon2) - \tfrac12\log(\sqrt{4\lambda_i^2+\tfrac{\varepsilon^2}{4}}+\tfrac\varepsilon2)\right] \\&+ \sum_i\left[ \tfrac12\sqrt{4\lambda_i^2+\tfrac{\varepsilon^2}{4}} - \sqrt{4\lambda_i^\star\lambda_i+\tfrac{\varepsilon^2}{4}}\right] - \tfrac{1}{2}B_\varepsilon(\Sigma_\star, \Sigma_\star).
\end{align*}
Using the fact that, on \(\bR_+\), \(\tfrac\varepsilon2\log(\cdot+\tfrac\varepsilon2)\) is 1-Lipschitz, \(\sqrt{4(\cdot)^2+\tfrac{\varepsilon^2}{4}}\) is 2-Lipschitz, and \(\sqrt{4\lambda_i^\star(\cdot)+\tfrac{\varepsilon^2}4}\) is \(\frac{\lambda_i^\star}{\tilde\varepsilon}\)-Lipschitz, we get that the expression above is \(L_i\)-Lipschitz in \(\lambda_i\). Using the estimates on \(|\lambda_i(t)-\lambda_i^\star|\) and summing on \(i\), we get \eqref{eq:Seps-cv-rate}.
\end{proof}

We finish this case study with a remark about the optimality of the flow map.

\begin{remark}[Optimality of the flow map]
A continuity equation \(\dot\mu_t = \div(\mu_tv_t)\), granted \((v_t)_t\) is smooth enough, induces a flow map, \(\Phi:[0, +\infty)\times \bR^d\to \bR^d\) defined so that \(t\mapsto x_t \coloneqq \Phi(t, x)\) is the solution of \(\dot x_t = - v_t(x_t)\) with \(x_0 = x\). The flow can then be written \(\mu_t=\Phi(t, \cdot)_\#\mu_0\). For our evolution \eqref{eq:SWGF} under the assumption of this section, since the gradients of the potentials are affine and that we can diagonalize the covariances in a fixed basis, a simple computation shows that the flow map at \(t\) is given by \(\Phi(t, x) = \Sigma_t^{\frac12} \Sigma_0^{-\frac12}x\), which is precisely the Monge map between \(\mu_0\) and \(\mu_t\) \cite[Remark 2.29]{peyre2019computational}. This is an intuitive result given that the evolution behaves like in one-dimension and is monotone.

This observation remains interesting since multiple state-of-the-art generative models use such a flow map to transport a source distribution to a target (which is the long-time limit of the flow), such as diffusion models \cite{song2021scorebased} and flow matching models \cite{lipman2023flow}. A natural question from an optimal transport viewpoint is whether the limiting flow map \(\Phi(\infty, \cdot)\) solves the optimal transport problem, since this can give numerical methods to build Monge maps.
For the Fokker-Planck equation arising in Diffusion models, this has been proven true for a Gaussian source \cite[Theorem 3.1]{khrulkov2023understanding}, however the result does not hold for all initial measures \cite{lavenant2022flow}. Some works also build Monge maps from flow matching models, e.g.~\cite{kornilov2024optimal}.
\end{remark}

\section{Numerical illustrations} \label{sec:numerics}
In this section we simulate the flow numerically to observe its behavior and compare it to the theory. Firstly we expose the discretization scheme we used and prove its convergence to the flow, and then show the results of the experiments. The code is available online\footnote{\url{https://github.com/mhardion/gaussian_sinkhorn_flow}}.
\subsection{Explicit time-discretization scheme}
While in the general theory, the classical time-discretization scheme used to obtain existence (the JKO scheme) is time-implicit, for numerical purposes an explicit scheme is much simpler to implement. We will consider only centered Gaussians since the flow of the means has a closed-form expression. Taking a Taylor expansion in \eqref{eq:SWGF}, for some time step \(\tau > 0\), integer \(k\) and a smooth function \(\varphi\) we get
\begin{align*}
\int\varphi\dd\mu_{(k+1)\tau} &= \int(\varphi -\tau\langle\nabla\varphi, \nabla (f_{\mu_{k\tau}, \mu_\star} - f_{\mu_{k\tau}})\rangle)\dd\mu_{k\tau} + o(\tau)\\
&= \int \varphi\circ(\id-\tau\nabla (f_{\mu_{k\tau}, \mu_\star} - f_{\mu_{k\tau}})) \dd\mu_{k\tau} + o(\tau)\\
&= \int\varphi \dd(\id-\tau\nabla (f_{\mu_{k\tau}, \mu_\star} - f_{\mu_{k\tau}}))_\#\mu_{k\tau} + o(\tau)
\end{align*}
which yields naturally the recursive scheme
\begin{equation}\label{eq:discr-scheme}
\mu_{k+1}^\tau = (\id - \tau \nabla (f_{\mu_k^\tau, \mu_\star} - f_{\mu_k^\tau}))_\#\mu_k^\tau.
\end{equation}
Whenever \(\mu_k^\tau\) is Gaussian, since \(\nabla (f_{\mu_k^\tau, \mu_\star} - f_{\mu_k^\tau})\) is affine, \(\mu_{k+1}^\tau\) will also be in \(\mathcal{G} \), and thus all iterations will stay in \(\mathcal G\) if \(\mu_0\) is Gaussian.
Note that the scheme \eqref{eq:discr-scheme} can be rephrased over the covariances as
\begin{equation}\label{eq:discr-scheme-cov}
\Sigma_{k+1}^\tau = (\id - \tau G_{\Sigma_k^\tau})\Sigma_k^\tau(\id - \tau G_{\Sigma_k^\tau}),
\end{equation}
yielding a straightforward implementation in closed form. We now show that this discrete scheme (interpolated appropriately) converges to the flow when \( \tau \to 0 \).
\begin{lemma}\label{lem:discr-scheme}
Let \(\mu_0\in\mathcal G\). For \(\tau > 0\), consider the scheme \eqref{eq:discr-scheme} initialized by \(\mu_0^\tau = \mu_0\). For some \(T>0\), write \((\mu_t^\tau)_{t\in [0,T]}\) the piecewise constant interpolation of this scheme up until \(k = \lfloor \frac{T}{\tau}\rfloor\). Then \((\mu_t^\tau)_{t\in[0, T]}\) converges uniformly in Wasserstein distance to the Gaussian solution \((\mu_t)_t\) of \eqref{eq:SWGF}.
\end{lemma}
\begin{proof}
We consider the covariances given by \eqref{eq:discr-scheme-cov}. For a fixed \(T>0\), the iterates \((\Sigma_k^\tau)_{0\leq k\leq T/\tau}\) are bounded independently of \(\tau\): indeed, by \cref{lem:Gbounded}, we get \(\|\Sigma_{k+1}^\tau\|_{\mathrm{op}} \leq (1+\tau\frac{\lambda_{\mathrm{max}}(\Sigma_\star)}{\tilde\varepsilon})^2\|\Sigma_k^\tau\|_{\mathrm{op}}\) and thus \(\|\Sigma_k^\tau\|_{\mathrm{op}} \leq (1+\tau\frac{\lambda_{\mathrm{max}}(\Sigma_\star)}{\tilde\varepsilon})^{2T/\tau}\|\Sigma_0\|_{\mathrm{op}}\) which converges to \(\exp(2T\lambda_{\mathrm{max}}(\Sigma_\star)/\tilde\varepsilon)\) as \(\tau \to 0\) and is thus bounded by a constant independent of \(\tau\) when it is small enough. Additionally, define the ``one-step method'' \(\Phi: (\Sigma, \tau) \mapsto -(G_\Sigma \Sigma + \Sigma G_\Sigma) + \tau G_\Sigma\Sigma G_\Sigma\) on \(\mathbb S_+^d \times [0, +\infty)\), so that \(\Sigma_{k+1}^\tau = \Sigma_k^\tau + \tau \Phi(\Sigma_k^\tau, \tau)\). This map can be shown to be locally Lipschitz in the first variable (thanks to the regularization allowing to stay away from the singularities in which the matrix square root and inverse are not Lipschitz), and it satisfies \(\Phi (\Sigma, 0) = -(G_\Sigma \Sigma + \Sigma G_\Sigma)\). From standard numerical analysis arguments (see e.g.~\cite[Chapter 5]{gautschi2012numerical}), we get \(\sup_{0\leq k\tau \leq T}\|\Sigma_{k\tau}-\Sigma_k^\tau\| = O(\tau)\) where \((\Sigma_t)_t\) is the solution of the evolution \eqref{eq:Sigmaflow}. Using for all \(t\in [k\tau, (k+1)\tau)\) that \(\|\Sigma_{t}-\Sigma_k^\tau\| \leq \|\Sigma_{k\tau}-\Sigma_k^\tau\| + \|\Sigma_{k\tau}-\Sigma_t\|\) and that the last term is \(O(\tau)\) because \(\dot\Sigma\) is continuous and thus bounded on \([0, T]\), we get locally uniform convergence of the piecewise constant interpolation \((\Sigma_t^\tau)_t\) to \((\Sigma_t)_t\). This gives locally uniform convergence of the measures to the limit curve in Wasserstein distance.
\end{proof}

\subsection{Qualitative behavior}
In \cref{fig:ellipses-non-singular}, we take a non-singular source and observe the behavior for a non-singular target and a singular one. It is already apparent that, as in the theory, the convergence is slower in the latter case, which will be illustrated quantitatively next section. The intuition is that the ``repulsive force'' \(\nabla f_{\mu_t}\) is stronger the more concentrated \(\mu_t\) is, so that convergence is slowed down when converging toward a Gaussian distribution with singular covariance, that is a measure fully concentrated along an axis. 
\begin{figure}[H]
\centering
\includegraphics[width=0.9\linewidth]{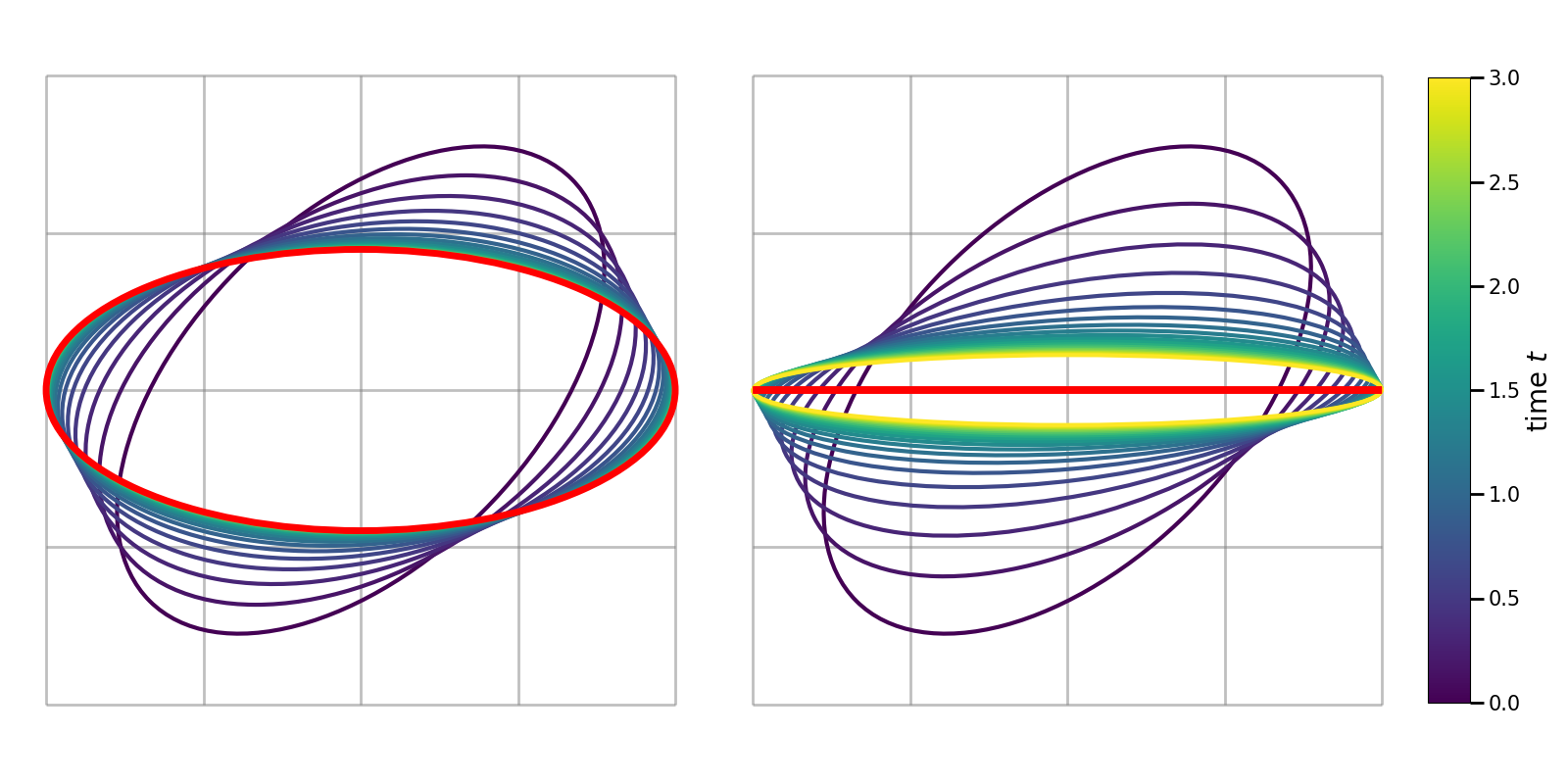} 
\caption{Covariance ellipses of the flow for a non-singular source and a non-singular (left) vs singular (right) target (red). The gray grid lines are spaced by \(\sqrt{\varepsilon}\).}
\label{fig:ellipses-non-singular}
\end{figure}
In \cref{fig:ellipses-singular}, we take a singular source and a singular target (each concentrated on a 1D axis), in two different configurations: with the two axes being orthogonal (in particular, the matrices commute), and a slight perturbation where the axes are not orthogonal but very close to the first configuration. As predicted, in the former case the flow does not converge to the target measure but rather to the Dirac \(\delta_0\), since the eigenvalues evolve independently and stay at 0 if they start at 0. In the latter case however we observe convergence to the target, which shows an instability of the flow: a slight perturbation which breaks the symmetry of the orthogonal case yields a different limit point.
\begin{figure}[H]
    \centering
    \includegraphics[width=\linewidth]{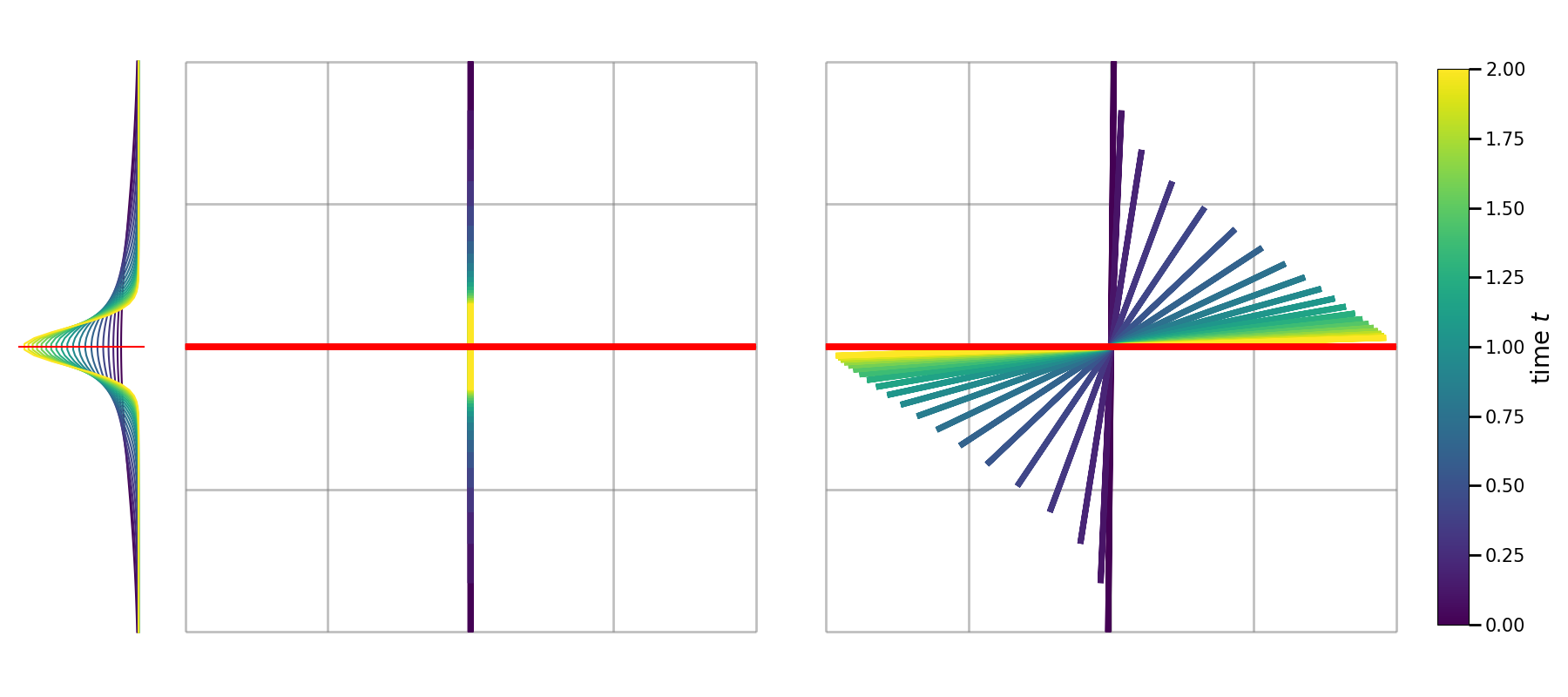}
    \caption{Covariance ellipses for singular Gaussian distributions, in an orthogonal configuration (left, with y-axis marginal for visual clarity) vs.~slightly rotated (right).}
    \label{fig:ellipses-singular}
\end{figure}
\subsection{Convergence rates}
We now move to a quantitative study of the convergence rates of the functional to 0 to empirically validate the tightness of the rates obtained in the commuting case in \cref{prop:eigflow}. In \cref{fig:cv-rate-lambda}, we compute the flow in commuting and non-commuting cases, for a target with diagonal covariance matrix and eigenvalues \( \{1,  \lambda^\star\} \) with $\lambda^\star$ decreasing from $0.5$ to $0$ (eventually making the target singular). We empirically observe that the rates are essentially unchanged whether the source and target commute or not, and that exponential convergence holds for non-singular measures (for large enough \(\lambda^\star\), it is a straight line in semi-log plot, for smaller but non-zero values the rate decreases in time but still ends up exponential as can be seen on the log-log plot), and sublinear convergence for \(\lambda^\star=0\) (we end up with a straight line in log-log plot), i.e. convergence in \(t^{-\alpha}\) for some \(\alpha >0\). It appears that \(\alpha\) is strictly greater than 1 in our simulations, so our estimate on the functional is likely pessimistic in terms of \(\alpha\) but seems accurate in the fact that convergence is sublinear rather than exponential.

\begin{figure}[H]
    \centering
    \includegraphics[width=\linewidth]{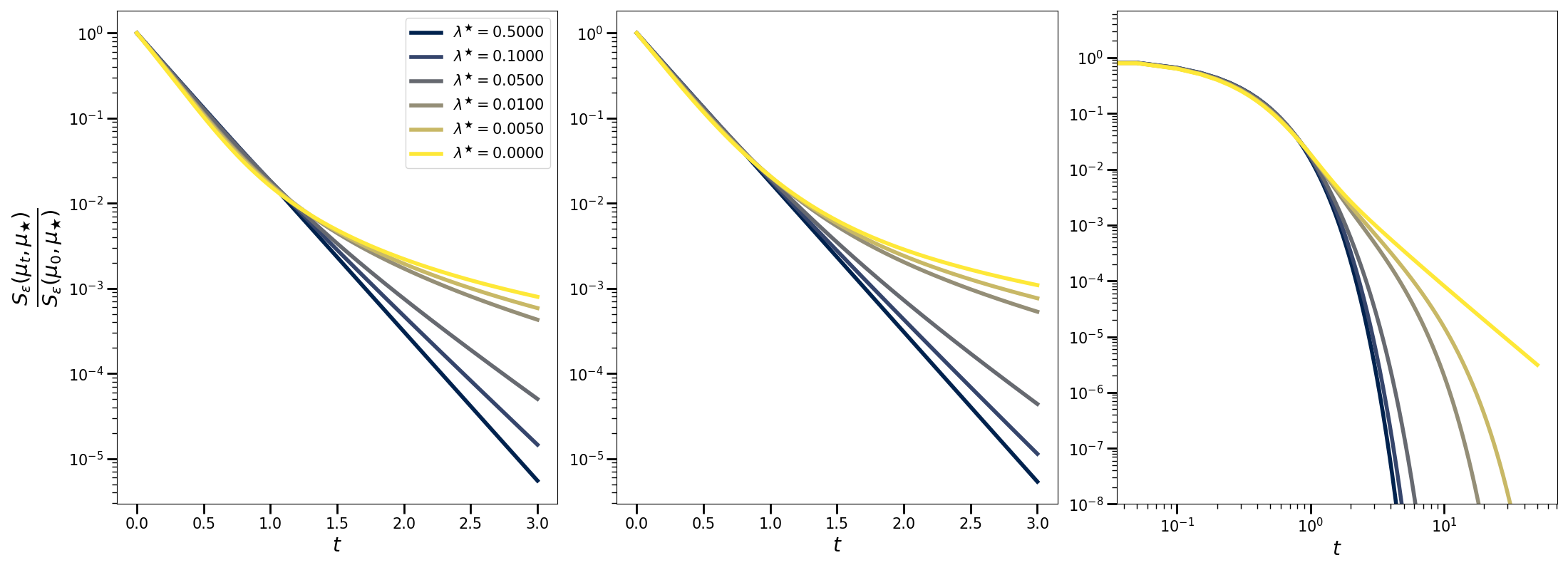}
    \caption{Values of \(S_\varepsilon(\mu_t, \mu_\star)/S_\varepsilon(\mu_0,\mu_\star)\) over time for \(\Sigma_0 =\id\) (commuting case, left) and \(\Sigma_0\) the same as in \cref{fig:ellipses-non-singular} (non-commuting case, middle), and \(\Sigma_\star =  \diag((1, \lambda^\star))\) for different values of \(\lambda^\star\), as a semi-log plot. On the right, the conditions are the same as the middle but for a longer time interval and in log-log scale.}
    \label{fig:cv-rate-lambda}
\end{figure}
Finally, we also look at the influence of the regularization parameter \(\varepsilon\) on the convergence rate. We can see \cref{fig:cv-rate-eps} that higher values of \(\varepsilon\) tend to yield slower rates, and indeed the expression of the constant \(C_i^a\) we found in \cref{prop:eigflow} goes to 0 as \(\varepsilon\to +\infty\). Note also that in the large \(\varepsilon\) regime, the Sinkhorn divergence converges toward the \emph{degenerate} Maximum Mean Discrepancy functional \( (\mu,\mu_\star) \mapsto 2 \left\| \int x \dd (\mu - \mu_\star)(x) \right\|^2 \), of which the gradient flow cannot globally converge. It is thus not surprising to observe a deterioration of the convergence rates for large \(\varepsilon\). 
When it comes to small values of \(\varepsilon\) however, our constants are pessimistic since they also go to 0 as \(\varepsilon \to 0\) whereas the simulations show that the convergence rates improve for smaller \(\varepsilon\).
\begin{figure}[H]
    \centering
    \includegraphics[width=0.5\linewidth]{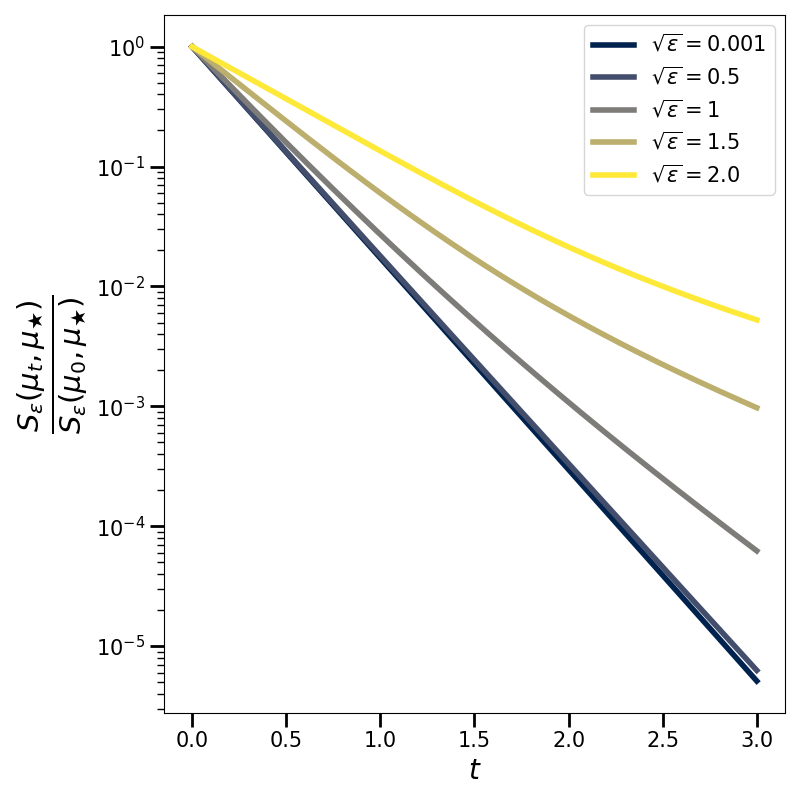}
    \caption{Values of \(S_\varepsilon(\mu_t, \mu_\star)/S_\varepsilon(\mu_0,\mu_\star)\) over time for different values of \(\varepsilon\) (same source and target as in the left of \cref{fig:ellipses-non-singular}).}
    \label{fig:cv-rate-eps}
\end{figure}

\subsection{Empirical analysis beyond the Gaussian setting} \label{sec:numerics:swiss-roll}

To conclude this numerical section, we explore whether the deterioration of the convergence rate when the target is supported on a lower-dimensional subset of the source's support still occurs in non-Gaussian settings. In order to test this hypothesis, we will take \(\mu_\star\) to be the ``Swiss roll'' distribution (see \cref{fig:swiss-roll}), which is supported on a 2-dimensional manifold in a 3D ambient space. Our goal is to compare the convergence rates between two cases: when the flow is computed in the ambient 3D space, from a source that has full 3D support to the 2D manifold supported target, versus when the flow is computed in the ``latent'' 2D space given by the natural coordinates of the manifold, from a 2D source to the 2D target. To make the comparison fair, in both cases the source will be an isotropic Gaussian measure that has the same total variance and mean as the target in the considered space, and the value of \(\varepsilon\) will be normalized as a twentieth of the total variance of the target distribution (approximated over 10,000 samples) divided by the dimension of the considered space.

\begin{figure}[ht]
    \centering
    \includegraphics[width=0.4\linewidth]{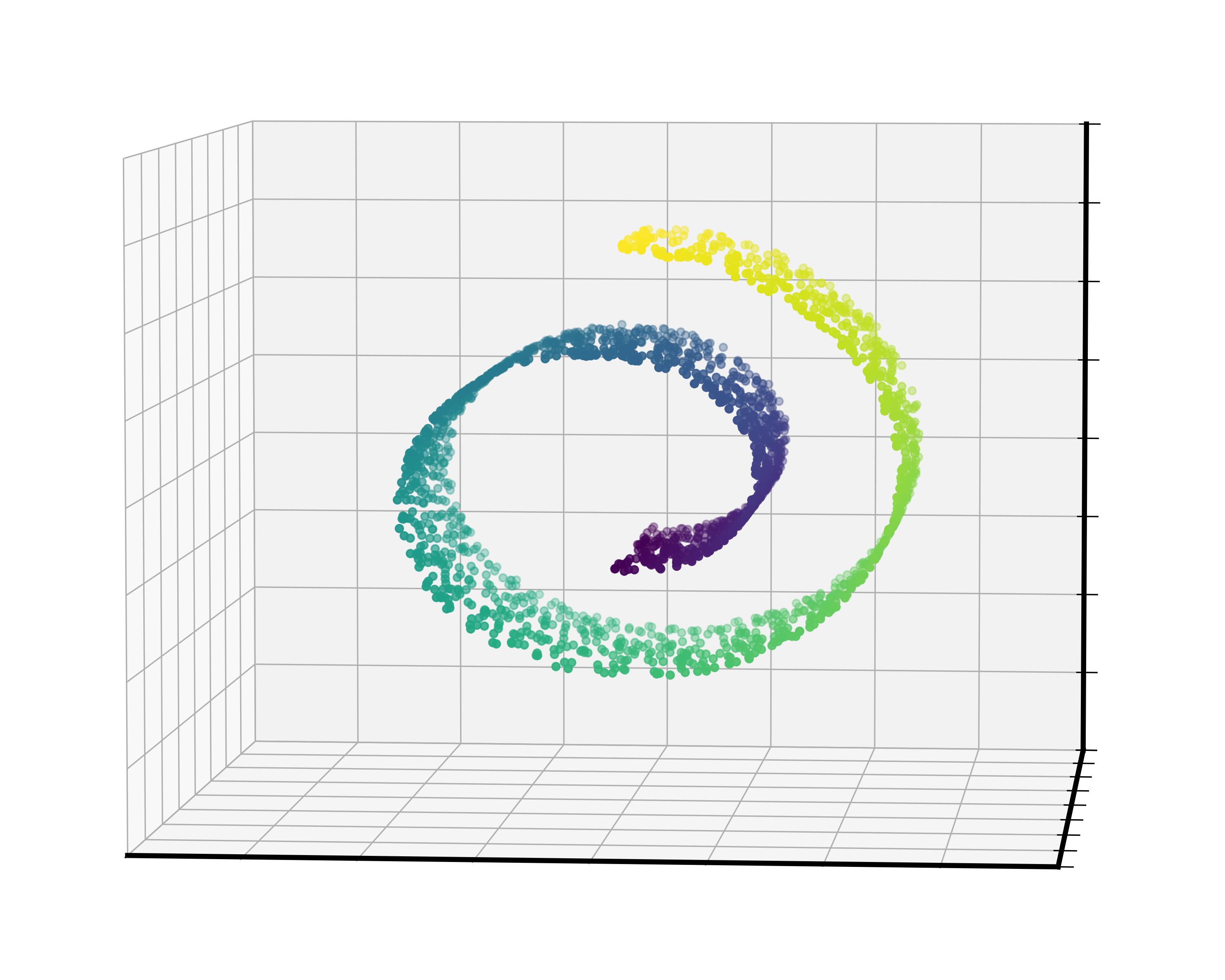}
    \caption{Samples of the target distribution, the ``Swiss roll''.}
    \label{fig:swiss-roll}
\end{figure}

Since a solution of \eqref{eq:SWGF} is not available in closed form, we will use a method based on drifting models \cite{deng2026generative} with drifting field given by the Wasserstein gradient of the Sinkhorn divergence, as recently done in \cite{he2026sinkhorndrifting, han2026onestep}. The idea is to parametrize \(\mu_{\theta} \coloneqq (\psi_\theta)_\#\tilde\mu\) where \(\tilde \mu\) is an easy-to-sample distribution and \(\psi_\theta\) is a neural network with parameters \(\theta\), then to train \(\psi_\theta\) iteratively so that \(\psi_{\theta_{t+\tau}} \simeq (\id- \tau \nabla (f_{\mu_{\theta_t}, \mu_\star} - f_{\mu_{\theta_t}}))\circ\psi_{\theta_t}\). This is done by minimizing at each \(t=k\tau\) the loss \(\mathcal L:\theta\mapsto \int\|\psi_{\theta} - (\id- \tau \nabla (f_{\mu_{\theta_t}, \mu_\star} - f_{\mu_{\theta_t}}))\circ\psi_{\theta_t}\|^2\dd\tilde\mu\) estimated through its Monte-Carlo counterpart in practice. We provide the full details of the training procedure and the parameters in \cref{sec:drifting-details}. If the model \(\psi_\theta\) is sufficiently expressive (i.e.~\(\argmin_{\theta}\mathcal L(\theta) \simeq 0\)) and the training in \(\theta\) goes well, we should thus have \(\mu_{\theta_{t+\tau}}\simeq (\id- \tau \nabla (f_{\mu_{\theta_t}, \mu_\star} - f_{\mu_{\theta_t}}))_\#\mu_{\theta_t}\) (similarly to \eqref{eq:discr-scheme}), so that \((\mu_{\theta_t})_t\) should be close to a solution of \eqref{eq:SWGF}.

After obtaining the sequence of trained networks \((\psi_{\theta_t})_{t=0, \tau, \dots}\), for a batch size \(N=1{,}000\), we independently sample \((X_t^i)_{i=1}^N\sim \mu_{\theta_t}\) (by sampling \((Z_t^i)_i\sim\tilde \mu\) i.i.d.~and computing \(X_t^i = \psi_{\theta_t}(Z_t^i)\)), and write \(\hat \mu_{\theta_t} = \frac1N\sum_{i=1}^N\delta_{X_t^i}\) the corresponding empirical measure. Similarly we write \(\hat\mu_\star=\frac1N\sum_{i=1}^N\delta_{Y_i}\) for \((Y_i)_{i=1}^N\) i.i.d.~samples of \(\mu_\star\). We then record the mean and standard deviation of \(S_\varepsilon(\hat\mu_{\theta_t}, \hat\mu_\star)\) over 10 batches of \(N\) samples.
Due to the stochastic nature of the procedure, even if we achieved perfect reconstruction i.e.~\(\mu_{\theta_t}= \mu_\star\) for some \(t\), then in distribution, \(S_\varepsilon(\hat\mu_{\theta_t}, \hat\mu_\star)\) would be the same as \(S_\varepsilon(\hat \mu_\star, \hat\mu'_\star)\) where \(\hat\mu_\star' \coloneqq \frac{1}{N}\sum_{i=1}^N\delta_{Y'_i}\) with \((Y'_i)_{i=1}^N\sim\mu_\star\) another i.i.d.~batch of samples independent from the \(Y_i\)s. Since \(S_\varepsilon(\hat \mu_\star, \hat\mu'_\star)\) is not 0 in general, we report the mean and standard deviation of this baseline value as well, since it is the best we can hope to achieve. The results are displayed in \cref{fig:losses}, and samples of the learned flow are presented in \cref{fig:flows-swiss-roll}.

\begin{figure}[ht]
    \centering
    \includegraphics[width=\textwidth]{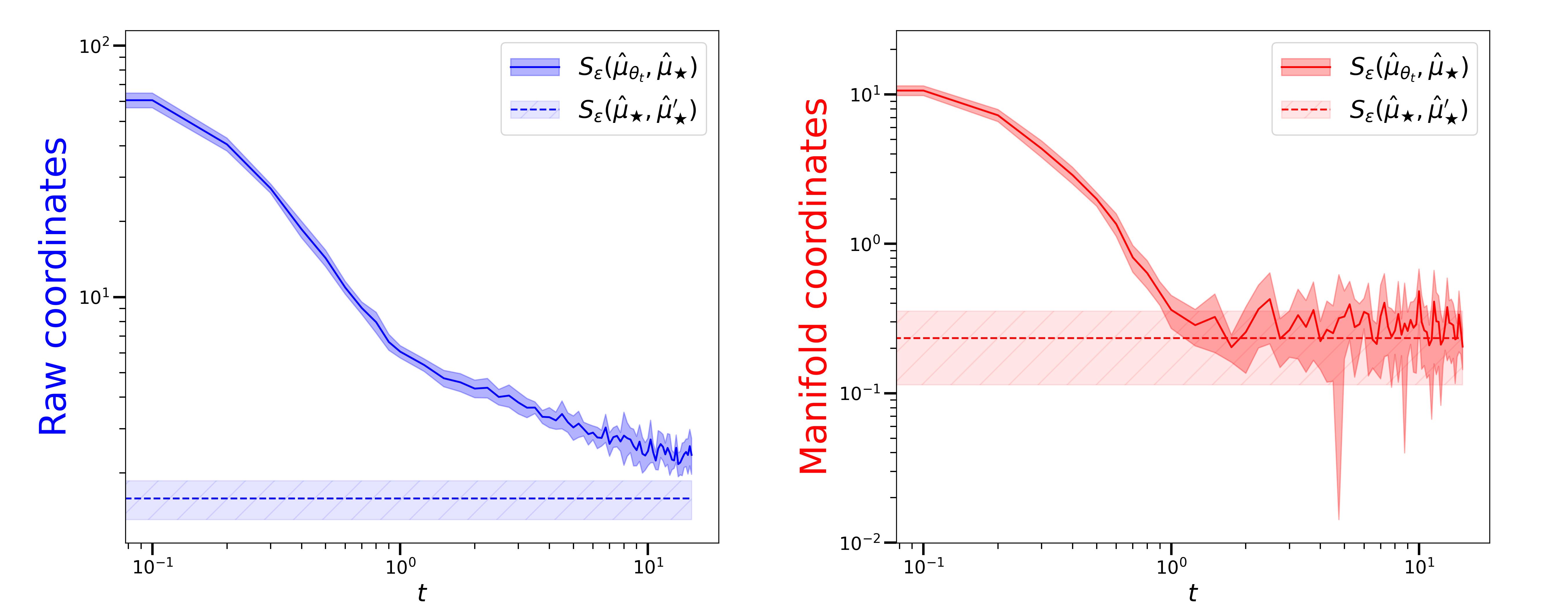}
    \caption{Values of \(S_\varepsilon\) over the flows (solid line: mean value, shaded band: one standard deviation, computed over 10 different batches of 1,000 samples). The scale is log-log.}
    \label{fig:losses}
\end{figure}

\begin{figure}[ht]
    \centering
    \includegraphics[width=\textwidth]{flows-swiss-roll.jpg}
    \caption{Samples of the models approximating the flows over time, represented in 3D space.}
    \label{fig:flows-swiss-roll}
\end{figure}

One can see in \cref{fig:losses} that in the raw coordinates, the convergence from the 3D Gaussian to the Swiss roll is slowed down to a sublinear-looking rate, whereas when both source and target are in the 2D coordinates of the manifold, convergence appears exponential and the baseline is quickly reached. We have carried out the procedure with various sizes of networks in both cases and have observed that the network size does not affect this phenomenon (see \cref{sec:drifting-details}). This suggests that the rates announced in \refthmitem{thm:cv-criterion}{commute} could hold in a more general setting, and that the need to collapse mass onto a manifold of lower dimension slows down convergence. These observations point toward the use of latent spaces, which can provide an embedding of the target into a space matching its ``natural'' dimension, thus making the flow starting from a fully supported Gaussian measure in that space (which is the most common type of source in practice) converge faster.

\section{Conclusion and perspectives}
In this paper, we have investigated the Wasserstein gradient flow of the Sinkhorn divergence to a target Gaussian measure, when the source is also Gaussian. After proving the well-posedness of this flow, we have shown that there are cases where convergence to the target holds, and cases where the limit is different from the target. We also derived convergence rates for the functional in the case where the source and target covariances commute, which are slower when the target is ``more singular'' than the source. These rates still seem to hold empirically for non-commuting matrices and even for some non-Gaussian distributions as illustrated by numerical experiments.

We hope that this work can be a first stepping stone toward a convergence criterion in more general settings. We conjecture that such a criterion could be related to the existence and uniqueness of Monge maps between the source and the target distributions. Indeed, so far we only see two possible configurations when convergence is obstructed, in both cases because of the impossibility of the flow to split mass due to the regularity of the Schr\"odinger potentials. The first one is when the source is ``more singular'' than the target, so that there is no map transporting the source to the target, and in particular since the flow is regular enough to induce a map it cannot transport the source to the target either. The second case is when there is some symmetry (as in \cref{fig:ellipses-singular}, left), as intuitively the mass is  being ``pulled'' equally on all sides, but not enough to split, preventing it from converging to the target and instead making it go toward an unstable point. Such a symmetry would correspond to the non-uniqueness of Monge maps. 

We also expect that some of our theoretical results and numerical observations hold in a broader setting, i.e.~it seems plausible that global convergence of the flow holds when the source admits a density with respect to the Lebesgue measure on \( \bR^d \), but that convergence rates drop when the target measure has no density, as suggested by our empirical investigation in \cref{sec:numerics:swiss-roll}.

Our analysis also heavily relies on the use of the quadratic cost on $\bR^d$, which has the drawback of yielding a degenerate functional when $\varepsilon \to +\infty$. Studying the Wasserstein gradient flow of Sinkhorn divergences for other ground costs might also be of interest, especially when relying on large values of $\varepsilon$ (which are amenable in computations).

\section*{Acknowledgements}
This work was funded by l'Agence Nationale de la Recherche (ANR) under grant ANR-24-CE23-7711 (TheATRE).

\appendix

\section{Pseudo-Riemannian geometry of the Wasserstein space}\label[appendix]{sec:pseudoriemW2}
This section gives a brief overview of how the Wasserstein space can be endowed with a natural Riemannian-like differential structure, see also \cite{otto2001geometry}, \cite[Chapter 8]{ambrosio2008gradient}, \cite[Section 7]{ambrosio2013users}. The first consideration is that absolutely continuous curves in this space are characterized as solutions of a continuity equation \(\dot \mu_t + \div(\mu_t v_t) = 0\) with some \(v_t \in L^2_{\mu_t}(\bR^d)\) for a.e.~\(t\) \cite[Theorem 8.3.1]{ambrosio2008gradient}. 

Thus a ``perturbation'' of \(\mu\) in the Wasserstein space, that is a tangent vector at \(\mu\), can be identified with a vector field \(v\in L^2_{\mu}(\bR^d)\). To be more precise, there may be multiple tangent vector field curves \((v_t)_t\) describing the same evolution of mass, so the tangent vector is taken to have minimal \(L_\mu^2(\bR^d)\)-norm, which can be characterized as being the gradient of a \(\mathcal C^\infty_c(\bR^d)\) function. Additionally, consider the celebrated Benamou-Brenier formula \cite{benamou2000computational}, which writes
\begin{equation*}
W_2^2(\mu, \nu) = \inf_{(\mu_t)_t, \:(v_t)_t}\int_0^1\|v_t\|^2_{L^2_{\mu_t}(\bR^d)}\dd t
\end{equation*}
where the infimum is taken along curves \((\mu_t)_{t\in[0, 1]}, \:(v_t)_{t\in[0,1]}\) satisfying the continuity equation \(\dot\mu_t + \div(\mu_tv_t) = 0\) with \(\mu_0 = \mu\), \(\mu_1=\nu\).
Comparing this formula with the definition of the geodesic distance on a Riemannian manifold \((\mathcal M, \mathsf g)\), i.e.
\begin{equation*}
\mathsf d(x, y)^2 = \inf_{(\gamma_t)_t} \int_0^1 \mathsf g_{\gamma_t}(\dot\gamma_t, \dot\gamma_t)\dd t
\end{equation*}
under the constraints \(\gamma_0 = x, \: \gamma_1=y\), we see that the Wasserstein space is a geodesic space akin to a Riemannian manifold with metric tensor given by the inner product in \(L^2_{\mu}(\bR^d)\) (this can also be seen with a Taylor expansion of the Wasserstein distance \cite[Proposition 8.5.6]{ambrosio2008gradient}).

If we restrict ourselves for a moment to measures \(\mu\) and \(\nu\) with Lebesgue density, so that we always have a unique Monge map \(\mathcal T\) (that is, a map such that \((\id, \mathcal T)_\#\mu \in \Pi_{\mathrm o}(\mu, \nu)\)), we can write the geodesic between these two measures as \(\mu_t = (\id + t(\mathcal T-\id))_\#\mu\), which can be shown to have a tangent vector field at \(t=0\) given by \(v_0 = \mathcal T - \id\). So \(\mathcal T - \id\) is the ``displacement vector'' from \(\mu\) to \(\nu\) (i.e.~ the Riemannian logarithm at \(\mu\) of \(\nu\)), where in a Hilbert space \(\mathcal H\), \(y-x\) would be the displacement vector from \(x\in\mathcal H\) to \(y\in\mathcal H\). In the latter setting, the Fr\'echet subdifferential of a functional \(F\) at \(x\in \mathcal H\) is defined as the set of \(p\in\mathcal H\) such that
\begin{equation*}
\forall y\in\mathcal H,\: F(y) - F(x) \geq \langle p, y-x\rangle_{\mathcal H} + o(\|y-x\|_{\mathcal H}).
\end{equation*}
Therefore, in our analogy, replacing \(y-x\) by \(\mathcal T - \id\) and replacing the inner product of \(\mathcal H\) with that of the tangent space, we can define the Wasserstein subdifferential of \(F:\mP_2(\bR^d) \to \bR\) at \(\mu\in\mP_2(\bR^d)\) as the set of \(\xi \in L^2_{\mu}(\bR^d)\) such that
\begin{equation*}
\forall \nu\in\mP_2(\bR^d),\: F(\nu) - F(\mu) \geq \int \langle\xi, \mathcal T-\id\rangle\dd\mu + o(W_2(\mu,\nu))
\end{equation*}
where we also have \(W_2(\mu,\nu) = \|\mathcal T-\id\|_{L^2_\mu(\bR^d)}\).
Note that with \(\pi \coloneqq (\id, \mathcal T)_\#\mu\), this inequality writes as \eqref{eq:defsubdiff}, so that we can drop the absolute continuity assumption and obtain the most general definition.

\section{On Bures--Wasserstein (BW) gradient flows}\label[appendix]{sec:BWGFs}

In this subsection, we use the following notations:
\begin{itemize}
    \item For \(\Sigma \in \mathbb{S}^d_{+}\), we let \(L_\Sigma : X \mapsto X \Sigma + \Sigma X\).
    \item For \(\Sigma \in \mathbb{S}^d_{++}\), and \(X,Y\) symmetric we let \(\mathbf g_\Sigma(X,Y) \coloneqq \frac12\mathrm{tr}(X L_\Sigma^{-1}(Y))\) denote the Bures--Wasserstein metric.
    \item For a regular functional \(F : \mathbb{S}^d_+ \to \mathbb{R}\), \(\nabla F(\Sigma)\) denotes its Euclidean gradient at \(\Sigma\) and \(\nabla_{\mathrm{BW}} F(\Sigma)\) denotes its BW gradient. Namely, one has for a symmetric matrix \(X\) such that \(\Sigma+X\in\mathbb S^d_+\) that
    \begin{equation*}
        F(\Sigma + X) - F(\Sigma) = \mathrm{Tr}(\nabla F(\Sigma) X) + o(X) = \frac12\mathrm{tr}(L_\Sigma^{-1}(\nabla_{\mathrm{BW}} F (\Sigma)) X) + o(X).
    \end{equation*}
    In particular, this implies that for any \(\Sigma \in \mathbb{S}^d_{++}\)
        \begin{equation}\label{eq:relation-BW-grad-Eucl-grad}
        \nabla_{\mathrm{BW}} F(\Sigma) = 2L_\Sigma(\nabla F(\Sigma)),
    \end{equation}
    and, by density of \(\mathbb{S}^d_{++}\) in \(\mathbb{S}^d_+\) and continuity of \(\Sigma \mapsto L_\Sigma(\nabla F(\Sigma))\), the relation extends to possibly singular covariance matrices. 
\end{itemize}

We now provide the following lemma, which lifts BW gradient flows as Euclidean gradient flows through the usual identification \(\Sigma = X^T X\) for some \(X \in \bR^{d \times d}\). 
This allows us to apply standard results for Euclidean gradient flow to the non-Euclidean setting of BW gradient flows. 

\begin{lemma}\label{lem:BW-flow-cv}
    Let \((\Sigma_t)_t\) be a BW gradient flow, i.e.~a curve valued in \(\mathbb S^d_+\) satisfying \(\dot \Sigma_t = - \nabla_{\mathrm{BW}} F(\Sigma_t)\) for some \(F : \mathbb{S}^d_+ \to \mathbb{R}\), assumed sufficiently regular to ensure existence and uniqueness of such flows when initialized at \(\Sigma_0 \in \mathbb{S}^d_+\). Assume also that \((\Sigma_t)\) is bounded. 
    Let \(\tilde F : \mathbb{R}^{d \times d} \to \mathbb{R}\) be defined as \(\tilde F(X) = F(X^T X)\), and let \((X_t)_t\) be the curve defined as the solution of the ODE 
    \begin{equation*}
        \dot X_t = - 2X_t \ \nabla F(X_t^T X_t),
    \end{equation*}
    with \(X_0\) such that \(X_0^T X_0 = \Sigma_0\). 
    Then
    \begin{thmenum}
        \item For all \(t\), \(X_t^T X_t = \Sigma_t\), \label{lem:BW-flow-cv:XTX=Sig}
        \item The curve \((X_t)_t\) is the solution of the \emph{Euclidean} gradient flow \(\dot X_t = -  \nabla \tilde F(X_t)\). \label{lem:BW-flow-cv:Xt-Eucl-GF}
        \item If \(F\) (and thus \(\tilde F\)) is analytic, then \((X_t)_t\) admits a limit point \(X_\infty\) as \(t \to +\infty\), and so does \((\Sigma_t)_t\). \label{lem:BW-flow-cv:limit}
    \end{thmenum}
\end{lemma}

\begin{proof} 
    Observe that
    \begin{equation*}
        \frac{\dd}{\dd t} (X_t^T X_t) = - 2(\nabla F(X_t^T X_t) X_t^T X_t + X_t^T X_t \nabla F(X_t^T X_t)) = - \nabla_{\mathrm{BW}} F(X_t^T X_t)
    \end{equation*}
    thanks to \eqref{eq:relation-BW-grad-Eucl-grad}. Therefore, \((X_t^T X_t)_t\) is a BW gradient flow of \(F\) with initial condition \(\Sigma_0\), which is assumed to be unique, yielding \(X_t^T X_t = \Sigma_t\) for all \(t\), proving \ref{lem:BW-flow-cv:XTX=Sig}.

    For \ref{lem:BW-flow-cv:Xt-Eucl-GF}, we can compute the (Euclidean) gradient of \(\tilde F\) by writing
    \begin{align*}
        \tilde F(X+h) - \tilde F(X) &= F((X+h)^T (X+h)) - F(X^T X) + o(h)\\
        &= \Tr(\nabla F(X^T X)(X^T h + h^T X)) + o(h)\\
        &= 2 \mathrm{Tr} ( (X \nabla F(X^T X))^T h) + o(h),
    \end{align*}
    where we used that \(\nabla F(X^T X)\) is symmetric. This proves that \(\nabla \tilde F(X) = 2X \nabla F(X^T X)\), yielding \ref{lem:BW-flow-cv:Xt-Eucl-GF}.

    As \((\Sigma_t)_t\) is assumed to be bounded, so is \((X_t)_t\) since the square Frobenius norm of \(X_t\) is given by \(\mathrm{Tr}(\Sigma_t)\), and thus \ref{lem:BW-flow-cv:limit} is obtained using \cite[Theorem 2.2]{absil2005convergence}.
\end{proof}

\section{Complete implementation details for \texorpdfstring{\cref{sec:numerics:swiss-roll}}{Section  \ref{sec:numerics:swiss-roll}}}\label[appendix]{sec:drifting-details}
In \cref{alg:drift} below, we detail the training procedure that we have used.

\begin{algorithm}[H]
\caption{Drifting model training}\label{alg:drift}
\SetKwInOut{Input}{Input}\SetKwInOut{Output}{Output}
\Input{Initial \(\theta_0\), time step \(\tau>0\) and maximal time \(T>0\), number of inner steps \(K>0\), regularization parameter \(\varepsilon\), samplers for \(\tilde\mu\) and \(\mu_\star\)}
\Output{\((\psi_{\theta_t})_{t=0, \tau, \dots, T}\) such that \(((\psi_{\theta_t})_\#\tilde\mu)_t\) approximate the flow}
\For{\(t = 0, \tau, 2\tau, \dots, T-\tau\)}{

\(\theta \leftarrow \theta_t\)\;

\For{\(k = 1, \dots, K\)}{

\((Y_i)_{i=1}^N\sim \mu_\star\)\;

\(\hat\mu_\star \leftarrow \frac1N\sum_{i=1}^N\delta_{Y_i}\)\;

\((Z_i)_{i=1}^N, (Z'_i)_{i=1}^N \sim \tilde \mu\)\;

\(\hat\mu_{\theta_t} \leftarrow \frac{1}{N}\sum_{i=1}^N\delta_{\psi_{\theta_t}(Z_i)}\)\;

\(\hat\mu'_{\theta_t} \leftarrow \frac{1}{N}\sum_{i=1}^N\delta_{\psi_{\theta_t}(Z'_i)}\)\;

\(\forall i\in \{1,\cdots, N\}, \hat V_{\theta_t}(\psi_{\theta_t}(Z_i)) \leftarrow -\nabla (f_{\hat\mu_{\theta_t}, \hat\mu_\star} - f_{\hat\mu_{\theta_t}, \hat\mu'_{\theta_t}})(\psi_{\theta_t}(Z_i))\)\label{alg:drift:gradients}\; 

\(\hat {\mathcal L}(\theta)\leftarrow \frac1N\sum_{i=1}^N\|\psi_\theta(Z_i) - (\psi_{\theta_t}(Z_i) +\tau \hat V_{\theta_t}(\psi_{\theta_t}(Z_i)))\|^2\)\;

\(\theta \leftarrow\) {\fontfamily{lmtt}\selectfont Adam}\((\theta, \hat{\mathcal L}(\theta))\)\label{alg:drift:adam}\;
}

\(\theta_{t+\tau} \leftarrow \theta\)\;}

\end{algorithm}

Note that in line~\ref{alg:drift:gradients}, the self-transport term is estimated using two batches of samples rather than one, which yields a more accurate estimation as explained in \cite{han2026onestep}. In the same line, the gradients of the potentials of discrete measures are computed using the formula \(\nabla f_{\mu, \nu}(x) = 2x- 2\int y\exp(\tfrac{1}{\varepsilon}(f_{\mu, \nu}(x) + g_{\mu, \nu}(y)-\|x-y\|^2))\dd\nu(y)\) (consequence of differentiating \eqref{eq:schrodinger-syst}) which is tractable, as the potentials themselves are computed using Sinkhorn's algorithm in log domain \cite{feydy2019interpolating}. In line~\ref{alg:drift:adam}, we write {\fontfamily{lmtt}\selectfont Adam}\((\theta, \hat{\mathcal L}(\theta))\) the updated value of \(\theta\) after a step of the Adam optimization algorithm \cite{kingma2015adam}. We also emphasize that only \(\theta\) is modified in the inner loop and \(\theta_t\) is fixed, i.e.~we push a frozen state of the model by the drifting field and do a regression towards it.

Note that the first distribution output by this algorithm is \(\mu_{\theta_0}\), and since we would like it to be a certain \(\mu_0\) specified in \cref{sec:numerics:swiss-roll}, we must first find a \(\theta_0\) such that \(\mu_{\theta_0}\) is a good approximation of \(\mu_0\). This is done through a round of pre-training using the same algorithm, with the desired \(\mu_0\) as target and a random initialization of \(\theta\).

In our experiments, the measure \(\tilde \mu\) is a standard Gaussian in the input space, and \(\psi_\theta\) is taken to be a multi-layer perceptron with 2 hidden layers. Writing the size of such a neural network as a quadruplet with the number of neurons in the input, 2 hidden layers and output respectively, in \cref{sec:numerics:swiss-roll} the two networks have sizes (10, 250, 250, 3) in the raw space versus (10, 250, 250, 2) in the natural coordinates, but we have also tested (20, 800, 200, 3) versus (5, 200, 100, 2) and (5, 200, 100, 3) versus (20, 800, 200, 2), without much influence on the convergence rate (see the code repository \url{https://github.com/mhardion/gaussian_sinkhorn_flow} for the additional figures).
\printbibliography
\end{document}